\documentclass[11pt,reqno,a4paper]{amsart}
\usepackage[a4paper,left=35mm,right=35mm,top=30mm,bottom=30mm,marginpar=25mm]{geometry}

\usepackage{graphicx,color}
\usepackage{amsmath,amsfonts,amssymb,amsthm,mathrsfs,amsopn}
\usepackage{latexsym}

\usepackage{microtype}

\usepackage[colorlinks=true, citecolor=blue]{hyperref} 

\allowdisplaybreaks

\newtheorem{Theorem}{Theorem}[section]
\newtheorem{Lemma}[Theorem]{Lemma}
\newtheorem{Proposition}[Theorem]{Proposition}
\newtheorem{Corollary}[Theorem]{Corollary}

\theoremstyle{definition}
\newtheorem{Definition}{Definition}[section]

\newtheorem{Assumption}[Definition]{Assumption}
\theoremstyle{remark}
\newtheorem{Remark}[Definition]{Remark}

\numberwithin{equation}{section}

\newcommand{\abs}[1]{\left \lvert#1 \right \rvert}

\newcommand{\norm}[1]{\left\lVert#1\right\rVert}

\def\H{\mathcal H}

\def\R{\mathbb R}
\def\N{\mathbb N}

\newcommand{\dist}{\mathop{\mathrm{dist}}}

\def\Id{{\rm Id}}

\def\spt{{\rm spt}}

\DeclareMathOperator{\trace}{tr}

\DeclareMathOperator{\range}{range}

\newcommand\res{\mathop{\hbox{\vrule height 7pt width .3pt depth 0pt \vrule height .3pt width 5pt depth 0pt}}\nolimits}

\newcommand\FF{\mathbf{F}}

\def\Id{\mathrm{Id}}

\DeclareMathOperator{\Div}{div}

\title[The Area Blow Up set for elliptic surface energy functionals]{The Area Blow Up set for bounded mean curvature submanifolds with respect to elliptic surface energy functionals}

\author[G. De Philippis]{Guido De Philippis}
\address{G.D.P.: SISSA, Via Bonomea 265, 34136 Trieste, Italy}
\email{guido.dephilippis@sissa.it}

\author[A. De Rosa]{Antonio De Rosa}
\address{A.D.R.: Courant Institute of Mathematical Sciences, 251 Mercer Street, New York, NY 10012, USA}
\email{derosa@cims.nyu.edu}

\author[J. Hirsch]{Jonas Hirsch}
\address{J.H.:  Mathematisches Institut, Universit\"at Leipzig, Augustus Platz 10, D04109 Leipzig, Germany}
\email{hirsch@math.uni-leipzig.de}


\begin{document}

 \dedicatory{To Luis Caffarelli, for his 70th birthday.}

\begin{abstract}
In this paper we investigate the ``area blow-up'' set of a sequence of smooth co-dimension one manifolds whose first variation with respect to an anisotropic integral is bounded. Following the ideas introduced by White in \cite{White2016}, we show that this set has bounded (anisotropic) mean curvature in the viscosity sense. In particular, this allows to show that the set is empty in a variety of situations. As a consequence, we show boundary curvature estimates for two dimensional stable anisotropic minimal surfaces, extending the results of \cite{White1987}.
\end{abstract}

\maketitle

\section{Introduction}
Consider a sequence $(M_i)_i$ of $m$-dimensional varieties in a subset $\Omega\subset \R^{m+1}$ with mean curvature bounded by some $h<\infty$ and such that the boundaries have uniformly bounded measure in compact sets:
\[
    \limsup_{i\to\infty} \H^{m-1}(\partial M_i \cap K)  < \infty, \qquad \forall K \, \text{compact}.
\]
Let $Z$ be the set of points at which the areas of the $M_i$ blow up:
\[
  Z := \{ x\in \Omega: \text{ $\limsup_i \H^m(M_i\cap B_r(x)) = \infty$ for every $r>0$} \},
\]
i.e. $Z$ is the smallest closed subset of $\Omega$ such that the areas of the $M_i$
are uniformly bounded as $i\to\infty$ on compact subsets of $\Omega\setminus Z$.

In the recent paper \cite{White2016}, White finds natural conditions  implying  that $Z$ is empty. These results are useful since if $Z$ is empty, then
the areas of the $M_i$ are uniformly bounded on all compact subsets of $\Omega$. It follows that, up to subsequences, $M_i$ will converge in the sense of varifold to a limit varifold
of locally bounded first variation.

The main point of \cite{White2016} is to show that the set  $Z$ belongs to the class of $(m,h)$-sets. The notion of  $(m,h)$-set is a  generalization of the concept of an $m$-dimensional, properly embedded submanifold without boundary and with mean curvature bounded
by $h$ \footnote{In particular, in \cite{White2016}, it is shown that if $M$ is a smooth, properly embedded, $m$-dimensional submanifold without boundary,
then $M$ is an $(m,h)$-set if and only if its mean curvature is bounded by $h$.}. In particular these sets satisfy a maximum principle which often allows to show that they are empty.

The aim of this paper is to extend the aforementioned results proven in \cite{White2016} to co-dimension one manifolds (or, more in general, to co-dimension one varifolds) which are stationary with respect to a parametric integrand \(F\). 

 Referring to  Section \ref{sec:preliminaries} below for more details and definitions we simply recall here that  a parametric integrand is a even map \(F: \Omega \times \R^{m+1}\to \R^+\)  which is one homogeneous, even and  {\em convex} in the second variable. For a smooth \(m\)-dimensional manifold  \(M\subset \R^{m+1}\)  with  normal \(\nu_{M}\) we define for every open set \(\Omega\subset \R^{m+1}\)
 \[
\FF(M, \Omega)= \int_{M\cap \Omega} F(x,\nu_{M}) d\H^{m}.
\]
A smooth manifold  is then said to be $F$-stationary in \(\Omega\) (resp. \(F\)-stable)  if
$$
\frac{d}{dt}\FF\big ( \varphi_t(M),\Omega\big)\Big|_{t=0}=0 \qquad\Biggl(\text{resp. } \quad \frac{d^2}{dt^2}\FF\big ( \varphi_t(M),\Omega\big)\Big|_{t=0}\ge0\Biggr)
$$
for every $\varphi_t(x)=x+tg(x)$ one-parameter family of diffeomorphisms (for \(t\) small enough) generated by a vector field \(g\in C_c^1(\Omega,\R^{m+1})\). 

In this setting our main result reads as follows, see Theorem \ref{thm:area-blowupset} for the more general statement and Definition \ref{def:mh} for the definition of \((m,h)\)-sets with respect to a given integrand \(F\):
\begin{Theorem}\label{thm:mainintro}
	Given a sequence of \(F\)-stable  \(m\)-dimensional manifolds $(M_i)_i$ and $h>0$ such that
	\begin{equation*}
\limsup_{i}\H^{m-1}(\partial M_i\cap K)<+\infty.
	\end{equation*}
Then the area-blow up set
\[ 
Z:= \{ x \in \overline \Omega \colon \limsup_{i \to \infty} \H^m(M_i\cap B_r(x))=+\infty \text{ for every $r>0$ } \}
\]
is an $(m,h)$-set in $\Omega$ with respect to $F$. 
\end{Theorem}

Beside its intrinsic interest, our main motivation for Theorem \ref{thm:mainintro}  is that, in contrast to the case of the area functional,  for manifolds which are stationary with respect to parametric integrand,  no monotonicity formula is available, \cite{Allard74}. In particular, a local area bound of the form
\begin{equation}\label{estimate}
\H^{m}(M\cap B_r(x))\le C(M,m) r^m
\end{equation}
is not know to hold true. This prevents, a priori,  the possibility to establish the convergence of the rescaled surfaces \(M_{x,r}=(M-x)/r\) in order to study the local behavior of a stationary surface.  Note that, for (isotropic) minimal surface, \eqref{estimate} is a trivial consequence of the monotonicity formula.  

Using Theorem \ref{thm:mainintro}, we can prove boundary curvature estimates for two dimensional \(F\) -stable surfaces, see also Theorem \ref{thm:curvature estimates at the boundary} for a more general statement:

\begin{Theorem}\label{thm:boundaryintro}  Let \(\Omega\subset \R^3\) be uniformly  convex, \(F\) be a uniformly elliptic integrand and let  $\Gamma\subset \Omega$ be a $C^{2,\alpha}$ embedded curve.  Let $M$ be an $F$-stable, $C^2$ \(2\)-dimensional embedded surface in $\Omega$ such that $\partial M = \Gamma$. Then there exist a constant $C>0$ and a radius $r_1>0$ depending only on $F, \Omega, \Gamma$  such that 
\begin{equation*}
 \sup_{\substack{p\in \Omega \\ \dist(p, \Gamma)< r_1}} r_1 |A_M(p)| \le C.
\end{equation*}
where \(A_M\) is the second fundamental form of \(M\). Furthermore the constants are uniform as long as \(\Gamma\), \(\Omega\) and \(F\) vary in compact subsets of, respectively,  embedded \(C^{2,\alpha}\) curves,  uniform convex domains and uniformly convex \(C^2\) integrands. 
\end{Theorem} 

Let us conclude this introduction with a few remarks on the proof of the main results. To prove Theorem \ref{thm:area-blowupset}, we follow the proof of White in  \cite{White2016}, and we aim to show that if the blow up set is not an \((m,h)\)-set, than one can provide a vector field yielding a   negative first variation. This vector field is what in \cite{SolomonWhite} is called an \(F\)-decreasing vector field and its construction  seems to be possible only in co-dimension one, which is the reason for our restriction to this setting. The proof of the boundary curvature estimates will easily follow from~\cite{White1987}, once we can show that the mass density ratios
\[
\frac{\H^{2}(M\cap B_r(x))}{r^2}
\]
are bounded. In the interior we can rely on the extended monotonicity formula for \(2\)-dimensional varifolds with curvature in \(L^2\) (note that by stability one easily proves that locally \(|A|\in L^2\)). At the boundary we perform a rescaling  argument and we use  our assumption on \(\Omega\) to show that that the area blow up set of the sequence of rescaled surfaces must be contained in a wedge. Since Theorem \ref{thm:area-blowupset} implies that this is a \((2,0)\)-set, a simple maximum principle argument shows that it is empty, yielding the desired bound.

\subsection*{Organization of the paper} The paper is organized as follows: in Section \ref{sec:preliminaries} we recall some preliminary results and definitions and we compute the explicit formula for the first variation of a smooth manifold. In Section \ref{sec:mh} we give the definition of \((m,h)\)-sets, we show some of their properties and we prove Theorem \ref{thm:area-blowupset}, from which Theorem \ref{thm:mainintro} readily  follows. In Section \ref{sec:bound} we prove Theorem \ref{thm:curvature estimates at the boundary}, which implies Theorem \ref{thm:boundaryintro}.

 \subsection*{Acknowledgements}
The work of  G.D.P.  is supported by the INDAM-grant ``Geometric Variational Problems".

\section{Notation and preliminaries}\label{sec:preliminaries}
We work on an open set \(\Omega\subset \R^{m+1}\) and we set   $B_r(x)=\{y\in\R^{m+1}:|x-y|<r\}$, \(B_r=B_r(0)\)  and $B:=B_{1}(0)$. We will denote \(m\)-dimensional balls by  \(B^m_r(x)\) and  we set \(B_r^m=B_r^m(0)\) and \(B^m=B_1^m\). We also let \(\mathbb S^{m}\)  be the unit sphere in \(\R^{m+1}\).

 For a matrix  \(A\in \R^{m+1}\otimes\R^{m+1}\),  \(A^*\) denotes its transpose. Given   \(A,B\in \R^{m+1}\otimes\R^{m+1}\), we define \( A:B=\trace A^* B=\sum_{ij} A_{ij} B_{ij}\), so that \(|A|^2= A:A\).

\subsection*{Varifolds}
We denote by \(\mathcal M_+(\Omega)\) (respectively \(\mathcal M(\Omega,\R^n)\), \(n\ge1\)) the set of positive (resp. \(\R^n\)-valued) Radon measures on \(\Omega\). Given a Radon measure \(\mu\), we denote by \(\spt \mu\) its support. For a  Borel set \(E\),  \(\mu\res E\) is the  restriction of \(\mu\) to \(E\), i.e. the measure defined by \([\mu\res E](A)=\mu(E\cap A)\). For an \(\R^n\)-valued Radon measure \(\mu\in \mathcal M(\Omega,\R^n)\), we denote by \(|\mu|\in \mathcal M_+(\Omega)\) its total variation and we recall that, for all open sets \(U\),
\[
|\mu|(U)=\sup\Bigg\{ \int\big \langle \varphi(x) ,d\mu(x)\big\rangle\,:\quad \varphi\in C_c^\infty(U,\R^n),\quad \|\varphi\|_\infty \le 1 \Bigg\}.
 \] 
If \(\eta :\R^{m+1}\to \R^{m+1}\) is a Borel map and \(\mu\) is a Radon measure, we let \(\eta_\# \mu=\mu\circ \eta^{-1}\) be the push-forward of \(\mu\) through \(\eta\). 
An $m$-varifold  on \(\Omega\) is  a positive  Radon measure $V$ on $\Omega\times \mathbb S^{m}$ which is even in the \(\mathbb S^{m}\) variable, i.e. such that 
$$
V(A\times S)=V(A\times (-S))\qquad\mbox{for all $A\subset \Omega$, $S\subset \mathbb S^{m}$.}
$$
We will denote with  \(\mathbb V_m(\Omega)\)  the set of all \(m\)-varifolds on \(\Omega\).

Given a diffeomorphism $\psi \in C^1(\Omega,\R^{m+1})$, we define the push-forward of $V\in\mathbb V_m(\Omega)$ with respect to $\psi$ as the varifold $\psi^\#V\in \mathbb V_m(\psi(\Omega))$ such that
\begin{multline*}
\int_{G(\psi(\Omega))}\Phi(x,\nu)d(\psi^\#V)(x,\nu)\\
=\int_{G(\Omega)}\Phi\left(\psi(x),\frac{((d_x\psi(x))^{-1})^*(\nu)}{|((d_x\psi(x))^{-1})^*(\nu)|}\right)J\psi(x,\nu^\perp) dV(x,\nu),
\end{multline*}
for every $\Phi\in C^0_c(G(\psi(\Omega)))$. Here $d_x\psi(x)$ is the differential mapping of $\psi$ at $x$ and
\[
J\psi(x,\nu^\perp):=\sqrt{\det\Big(\big(d_x\psi\big|_{\nu^\perp}\big)^*\circ d_x\psi\big|_{\nu^\perp}\Big)}
 \]
 denotes the $m$-Jacobian determinant of the differential $d_x\psi(x)$ restricted to the $m$-plane $\nu^\perp$, 
 see \cite[Chapter 8]{Simon}.

\subsection*{Integrands}
The anisotropic (elliptic) integrands  that we consider are \(C^2\) positive functions
$$F:\Omega\times (\R^{m+1}\setminus\{0\})\to \R^+$$
which are even, one-homogeneous and convex in the second variable, i.e.
\begin{equation*}\label{e:identification}
F(x, \lambda  \nu)=|\lambda| F(x,  \nu) 
\end{equation*}
and
\[
F(x,\nu_1+\nu_2)\le F(x,\nu_1)+F(x,\nu_2).
\]
We will denote with $D_1F(x,\nu)$ and $D_2F(x,\nu)$ respectively the differential of $F$ in the first and in the second variable. Denoting with $\{e^x_i\}_{i=1}^{m+1}$ the euclidean basis in $\R_{x}^{m+1}$ and with $\{e^\nu_i\}_{i=1}^{m+1}$ the euclidean basis in $\R_{\nu}^{m+1}$, we set 
\begin{equation}\label{e:fj}
\begin{aligned}
&F_{i}(x,\nu):=\langle D_2F(x,\nu),  e^\nu_i \rangle, &&( \partial_i F_j)(x,\nu)=D_{12}F(x,\nu):e^x_i \otimes  e^\nu_j 
\\
&\qquad \text{ and }  &&F_{ij}(x,\nu):=  D^2_2F(x,\nu):e^\nu_i \otimes  e^\nu_j .
\end{aligned}
\end{equation}
Note that by one-homogeneity: 
\begin{equation}\label{eulercod1}
\langle D_2 F(x, \nu),\nu\rangle = F(x,\nu)\qquad\mbox{for all \(\nu\in\R^{m+1}\setminus \{0\}\).}
\end{equation}
An integrand \(F\) is said to be uniformly elliptic on a set \(\Omega\) if there exists a constant \(\lambda>0\) such that
\[
\langle D_2^2F(x,\nu)\eta,\eta\rangle\ge  \lambda|\eta|^2\qquad\text{for all \(x\in \overline{\Omega}\), \(\nu \in \mathbb{S}^m\), \(\eta\perp \nu\)}.
\]
Given \(x\in \Omega\), we will denote  by \(F_x\)  the ``frozen'' integrand 
\begin{equation*}\label{frozen}
F_x:\mathbb S^m \to (0,+\infty), \qquad F_x(\nu):= F(x,\nu).
\end{equation*}
We define the {\em anisotropic energy} of \(V\in \mathbb V_m(\Omega)\) as 
 \begin{equation*}\label{eq:eV}
\FF(V,\Omega) := \int_{G(\Omega)} F(x,\nu)\, dV(x,\nu). 
\end{equation*}

For a vector field \(g\in C_c^1(\Omega,\R^{m+1})\), we consider the family of functions \(\varphi_t(x)=x+tg(x)\), and we note that they are diffeomorphisms of \(\Omega\) into itself for \(t\) small enough. The {\em anisotropic first variation} is  defined as 
\[
\delta_F V(g):=\frac{d}{dt}\FF\big ( \varphi_t^{\#}V,\Omega\big)\Big|_{t=0}.
\]
It can be easily shown, see \cite[Appendix A]{DePhilippisDeRosaGhiraldin}, that
 \begin{equation}\label{eq:firstvariation}
\delta_F V(g)
=\int_{G(\Omega)} \Big[\langle D_1F(x,\nu),g(x)\rangle+ B_F(x,\nu):Dg(x)  \Big] dV(x,\nu),
 \end{equation}
where  the matrix \(B_F(x,\nu)\in \R^{m+1}\otimes\R^{m+1}\) is uniquely defined by 
\begin{equation}\label{eq:range of B}
B_F(x,\nu):=F(x,\nu)\Id-\nu \otimes D_2 F(x,\nu),
 \end{equation}  
 see for instance~\cite[Section 3]{Allard84BOOK} or ~\cite[Lemma A.4]{dephilippismaggi2}. We will often omit in the sequel the dependence on $F$ of the matrix $B_F(x, \nu)$. Moreover let us note the following useful fact:
 \begin{equation}\label{e:rangeB}
 B(x,\nu)\nu=0\qquad\text{or equivalently}\qquad \range B^*(x,\nu)=\nu^\perp
 \end{equation}

 We say that a varifold \(V\in \mathbb V_m(\Omega)\) has locally bounded anisotropic first variation if $\delta_F V$ is a Radon measure on $\Omega$, i.e. if
$$|\delta_F V(g)|\leq C(K)\|g\|_{\infty}, \quad \text{ for all $g \in  C^1_c(\Omega,\R^{m+1})$ with spt$(g)\subset K \subset\subset \Omega$}.$$
Notice that, by Riesz representation theorem, we can write
$$\delta_F V(g) = - \int_\Omega \langle w, g \rangle d \|\delta_F V\|, \quad \text{ for all } g \in C^1_c(\Omega, \R^{m+1})$$
where $\|\delta_F V\|$ is the total variation of $\delta_F V$ and $w$ is $\|\delta_F V\|$-measurable with $|w|= 1$ $\|\delta_F V\|$-a.e. in $\Omega$. 
In this case, by the Radon-Nikodym  theorem, we can decompose $\|\delta_FV\|$ in its absolutely continuous and singular parts with respect to the measure $\|V\|$:
\begin{equation}\label{rappr}
\delta_F V(g)= - \int_\Omega\langle \overline{H_F}, g \rangle \,d\|V\|(x)+\int_\Omega \langle w, g \rangle \,d\sigma, \quad \text{ for all } g \in C^1_c(\Omega, \R^{m+1}).
\end{equation}
Notice that by the disintegration theorem for measures, see for instance~\cite[Theorem 2.28]{AmbrosioFuscoPallara00}, we can write
\[
V(dx,d\nu)= \|V\|(dx)\otimes \mu_x(d\nu),
\]
where \(\mu_x\in \mathcal P(\mathbb S^m)\) is a  (measurable) family of parametrized  non-negative even probability measures. 
We define  for $\|V\|$-a.e. $x \in \Omega$
$$
H_F(x):=  \frac{\overline{H_F(x)}}{\int_{\mathbb S^m}F(x, \nu)d \mu_x(\nu)}.
$$ 
We will say that a varifold \(V\in \mathbb V_m(\Omega)\) has mean curvature $H_F(x)$ in $L^1(\norm{V}, \R^{m+1})$ if it has locally bounded anisotropic first variation and in the representation \eqref{rappr}, we have $\sigma=0$. In this case one can easily check that
\begin{equation}\label{eq:F-mean curvature 1} 
\delta_FV(g) = -\int_{G(\Omega)} \langle H_F, g \rangle \;F(x,\nu)\,dV(x,\nu) \text{ for all } g\in C^1_c(\Omega, \R^{m+1}).
 \end{equation}
Furthermore we will say that $H_F(x)$ is bounded by $h \in \R$ if 
$$\norm{H_F}_{F,x} := F(x,H_F(x))\le h.$$ 
In particular we say that a varifold \(V\in \mathbb V_m(\Omega)\) has anisotropic mean curvature bounded by $h(x) \in L^1(\norm{V},\R^+)$ if 
\begin{equation}\label{eq:F-mean curvature 2} 
\delta_FV(g) \le \int_{G(\Omega)} h(x) \norm{g}_{F^*,x} F(x,\nu)\,dV(x,\nu) \text{ for all } g \in C^1_c(\Omega, \R^{m+1}),
\end{equation}
where 
\[
\norm{w}_{F^*,x}= F^*(x, w)=\sup_{v: \, F(x,v)\le 1} \langle v,w\rangle.
\]

\begin{Remark}
Since all norms are equivalent on finite dimensional spaces, the above definition coincides with the classical one. However the above formulation has the advantage of being coordinate independent, namely if \(\Phi: \R^{m+1}\to \R^{m+1}\) is a diffeomorphism and \(V\) has \(F\)-mean curvature bounded by \(h\) then \(\Phi^{\#} V\) has  \(\Phi^{\#}F\)-mean curvature still bounded by \(h\) where \(\Phi^{\#}F\) is the integrand defined by
\[
\Phi^{\#}F(x,\nu)=F\left(\Phi^{-1}(x),(d_x\Phi(\Phi^{-1}(x)))^{*}(\nu)\right) \abs{\det(d_x\Phi^{-1}(x))}
\] 
and it satisfies
\[
\Phi^{\#}F(\Phi^{\#} V,\Phi(\Omega))=\FF(V,\Omega).
\]
 In particular we have $H_{\Phi^\# F}$ of the varifold $\Phi^\# V$ is $(d\Phi^*)^{-1}H_F$ where $H_F$ is the anisotropic mean curvature of the varifold $V$.
\end{Remark}

We conclude this section by computing the first variation formula for the varifold induced by a  manifold with boundary and by providing an explicit formula for its \(F\) mean curvature 

\begin{Proposition}
Let \(M\subset \R^{m+1}\)  be an oriented  $C^2$ $m$-manifold $M$ with boundary, and let 
\[
V_M:=\H^m \res M \otimes \left ( \frac 12\delta_{\nu_x}+\frac 12 \delta_{-\nu_x}\right ),
\]
where \(\nu_x\) is the normal to \(M\) at \(x\). Then 
\begin{equation}\label{eq:boundaryvariation}
\begin{split}
\delta_FV_M(g) = \int_{\partial M} \langle B(x,\nu_x)\eta(x),g(x) \rangle d \H^{m-1} -\int_{M} \langle H_F(x,M), g(x) \rangle \;F(x,\nu_x)d\H^m,
\end{split}
\end{equation}
for all $g\in C^1_c(\Omega, \R^{m+1})$. Here $\eta(x)$ denotes the conormal of $\partial M$ at $x$, \(H_F(x,M)\) is parallel to \(\nu_x\) and satisfies
\begin{equation}\label{eq:H_M}
	 -F(x,\nu_x)H_{F}(x,M)= \Bigl(D_2^2F(x,\nu):A+\sum_{i} (\partial_i F_i)(x,\nu)\Bigr)\nu_x.
\end{equation}
Here $A$ is the second fundamental form\footnote{Note that by this sign convention the second fundamental form is positive definite for a convex set with respect to the \emph{outer} normal.} of $M$ defined  by 
\[
A(\tau_1,\tau_2)= \langle \tau_1 , D_{\tau_2} \nu\rangle \text{for $\tau_1,\tau_2 \in T_x$}
\]
 and we are adopting the convention in \eqref{e:fj}.
\end{Proposition}
Note that  \eqref{eq:H_M} gives 
\[
\norm{H_F}_{F,x}= \abs{ \Bigl(D_2^2F(x,\nu):A+\sum_{i} (\partial_i F_i)(x,\nu)\Bigr)}.
\]
Moreover,  by \eqref{eq:H_M}  and the homogeneity of \(F\), if \(M=\{f=0\}\)  locally around \(x\) for a \(C^2\) function \(f\) with \(D f(x)\ne  0\), then 
\begin{equation}\label{calcolo}
\begin{split}
-F&(x,Df(x))\Big\langle H_F(x, M), \frac{Df(x)}{|Df(x)|} \Big\rangle \\
&= \trace\left(D^2_2F\left(x, \frac{Df(x)}{|Df(x)|}\right) D^2f(x)\right) + \sum_{i}(\partial_{i} F_i)\left(x, \frac{Df(x)}{|Df(x)|}\right)|Df(x)|.
\end{split}
\end{equation}

\begin{proof}
Recall that for a vector field \(X\) 
\[
\Div X=\Div_M X+\langle D_\nu X,\nu\rangle,
\]
where for any orthonormal basis $\tau_j$ of $T_xM=\nu^\perp$ one has
\[
\Div_M X=\sum_{i} \langle D_{\tau_i}X, \tau_i\rangle.
\]
Hence, if \(e_i\) is the standard orthonormal basis of \(\R^n\) and we adopt Einstein convention 
\begin{equation}\label{e:Bg}
\begin{split}
B:Dg&=\Div (B^* g)-\langle \Div B,g\rangle
\\
&=\Div_{M}(B^* g)- \Div_M (B^*e_i)g^i+\langle D_\nu(B^* g),\nu\rangle-\langle D_\nu(B^* e_i),\nu\rangle g^i
\\
&=\Div_{M}(B^* g)-\Div_M (B^*e_i)g^i
\end{split}
\end{equation}
where \(B\) is evaluated at $(x, \nu_x)$ and in the last equality  we used that $\langle \nu, B^* D_\nu g\rangle =0$ due to \eqref{eq:range of B}.  Note that \(B^*g\) is tangent to \(M\) (again by \eqref{eq:range of B}), hence by the divergence  theorem
\[
\begin{split}
\delta_FV(g) 
&=\int_{M} B:Dg+\langle D_1F(x,\nu), g\rangle
\\
&=  \int_{\partial M} \langle B(x,\nu) \eta, g \rangle - \int_{M} \left(\Div_M(B^*e_i) -\langle D_1F(x,\nu),e_i\rangle \right) g^i.
\end{split}
\]
Hence, if we set
\begin{equation}\label{e:HH}
F(x,\nu) H_F(x,M)=\Big(\Div_M(B^*e_i)- \langle D_1F(x,\nu),e_i\rangle \Big)e_i,
\end{equation}
the proof will be concluded, provided \(H_F(x,M)\) satisfies \eqref{eq:H_M}. This follows by direct computations since 
\begin{equation}\label{e:freddo1}
\begin{split}
\Div_M(B^*e_i) &= \langle \tau_j, e_i\rangle \left( F_k D_{\tau_j} \nu^k +\langle D_1F,\tau_j\rangle\right) \\
& \quad - \langle \tau_j, D_{\tau_j}(D_2F)\rangle \langle \nu, e_i\rangle - \langle \tau_k, D_2F\rangle \langle D_{\tau_k}\nu, e_i\rangle \\
&=\langle \tau_j, e_i\rangle \langle D_1F,\tau_j\rangle - \langle \tau_j, D_{\tau_j}(D_2F)\rangle \langle \nu, e_i\rangle, 
\end{split}
\end{equation}
where we used that $F_kD_{\tau_j}\nu^k = \langle \tau_h, D_2F\rangle \langle \tau_h, D_{\tau_j}\nu\rangle$ since $\langle \nu, D_{\tau_j}\nu\rangle =0$ and so
\begin{align*}
	\langle \tau_j, e_i\rangle F_k D_{\tau_j} \nu^k - &\langle \tau_k, D_2F\rangle \langle D_{\tau_k}\nu, e_i\rangle = \langle \tau_k , D_2 F\rangle \left( \langle \tau_j, e_i\rangle \langle \tau_k, D_{\tau_j}\nu\rangle - \langle D_{\tau_k}\nu, e_i \rangle \right)
\\
	&=\langle \tau_k , D_2 F\rangle \langle \tau_j, e_i\rangle \left(\langle \tau_k, D_{\tau_j}\nu\rangle - \langle \tau_j, D_{\tau_k}\nu\rangle\right) =0.
\end{align*}
Now we note that
\begin{equation}\label{e:freddo2}
\begin{split}
\langle \tau_j, e_i\rangle \langle D_1F,\tau_j\rangle&= \langle D_1F,e_i\rangle- \langle \nu, e_i\rangle \langle D_1F,\nu\rangle 
 \\
 &= \langle D_1F,e_i\rangle-   \langle \nu, e_i\rangle D_{12}F:\nu \otimes \nu, 
 \end{split}
\end{equation}
where in the last equality we have used the one-homogeneity of \(D_1F\). Furthermore
\begin{equation}\label{e:freddo3}
\begin{split}
	\langle \tau_j, D_{\tau_j}(D_2F)\rangle &=D_{12} F:\tau_j\otimes \tau_j+ D_2^2F(\tau_j, D_{\tau_j}\nu)
	\\
	&=  D_{12} F:\tau_j\otimes \tau_j+ D_2^2F(\tau_j, \tau_\ell)A_{\ell j},
\end{split}
\end{equation}
where \(A_{\ell j}=\langle D_{\tau_j}\nu,\tau_{\ell}\rangle\) is the second fundamental form of \(M\). Combining \eqref{e:freddo1}, \eqref{e:freddo2} and \eqref{e:freddo3},  we get \eqref{eq:H_M} since
\begin{align*}
\Div_M&(B^*(x,\nu)e_i)-\langle D_1F,e_i\rangle \\
&= - \langle \nu, e_i\rangle\Big( D_{12} F:\nu\otimes \nu+ \sum_{j}D_{12} F:\tau_j\otimes \tau_j+D_2^2F(\tau_j, \tau_\ell)A_{\ell j}\Big)
\\
&= - \langle \nu, e_i\rangle\left(\partial_j F_j+ \trace (D^2F A)\right),
\end{align*}
where  in the last equality we have used that, by \eqref{e:fj}),
\[
\partial_jF_j=\sum_{j}D_{12} F:e_j\otimes e_j=D_{12} F:\nu\otimes \nu+ \sum_{j}D_{12} F:\tau_j\otimes \tau_j\,.
\]
\end{proof}

\begin{Remark}\label{rmk:fdecrasing}
Let us record here the following consequence of the above computations: if \(X=D_2F(x, a(x)\nu_x)\) on \(M\) with $a\in C^1(M, \R_+)$, then \(B^* X=0\) and thus, by \eqref{e:Bg}, \eqref{e:HH}  we get
\begin{equation}\label{eq:good vectorfield identity -1}
 -\langle H_M(x,\nu), X\rangle \, F(x, \nu_x) = B(x, \nu_x):DX + \langle D_1F(x,\nu_x), X\rangle.
 \end{equation}
 \(X\) is what is called an \(F\)-decreasing vector filed in \cite[Proposition 1]{SolomonWhite} and it will play a crucial role in the proof of our main theorem.
\end{Remark}

\section{$(m,h)$-sets} \label{sec:mh}
In this section, following \cite{White2016},  we define  \((m,h)\)-sets and we prove that the area-blow up set of a sequence of varifolds with bounded curvature is an \((m,h)\)-set. Roughly speaking an \((m,h)\)-set is a set which can not be touched by manifolds  with \(F\langle H_F, \nu \rangle\) greater than \(h\), i.e. they satisfy \(\norm{H_F} \le h\) in the viscosity sense. This can be phrased in several ways, as the following proposition shows.

\begin{Proposition}\label{prop:equivalence}
Given a closed set $Z\subset \R^{m+1}$, then the following three statements are equivalent.
\begin{itemize}
\item[(i)] If $f: \Omega \to \R$ is a $C^2$-function and if $f|_Z$ has a local maximum at $p$, then 
\begin{equation*}
\inf_{v \in \mathbb S^m} F_{ij}(p,v) D_{ij}f(p)  + (\partial_{i} F_i)\left(p, \frac{D f(p)}{|D f(p)|}\right)|D f(p)| \le h \,\abs{Df(p)},
 \end{equation*}
 where the second term in the left hand side is intended to be zero when $D f(p) = 0$.
\item[(ii)]  If $f: \Omega \to \R$ is a $C^2$-function and if $f|_Z$ has a local maximum at $p$ and $D f(p) \neq 0$, then  
\begin{equation*}
 F_{ij}\left(p,\frac{D f(p)}{\abs{D f(p)}}\right) D_{ij}f(p)+ (\partial_{i} F_i)\left(p,\frac{D f(p)}{|D f(p)|}\right)|D f(p)| \le h\, \abs{Df(p)}.
 \end{equation*}
\item[(iii)] Let $N$ be a relative closed domain in $\Omega$ with smooth boundary, such that $Z\subset N$ and $p \in Z \cap \partial N$, then the $F$-mean curvature $H_F(p)$ of $\partial N$  satisfies
\[ F(p, \nu_{\textnormal{int.}}(p)) \langle H_F(p), \nu_{\textnormal{int.}}(p) \rangle \le h.\]
\end{itemize} 
where \(\nu_{\textnormal{int.}}\) is the \emph{interior} normal to \(N\).
\end{Proposition}

We can now give the following  definition

\begin{Definition}\label{def:mh}
Given an elliptic integrand \(F\) and and open set \(\Omega\) of \(\R^{m+1}\), we say that a relatively closed set set $Z\subset \Omega$ is an $(m,h)$-set with respect to \(F\) if it satisfies  one of the three equivalent conditions of Proposition \ref{prop:equivalence}.
\end{Definition}

Let us  prove Proposition \ref{prop:equivalence}.
\begin{proof}[Proof of Proposition \ref{prop:equivalence}]

\emph{ (ii) $\Rightarrow$ (iii):} This is an easy consequence of \eqref{calcolo} and of the elementary  Lemmas \ref{le1} and \ref{le2} below. Note that $\nu_{\text{int.}}(p)= - \frac{Df(p)}{\abs{Df(p)}}$ if $p \in \partial N$ and $N$ coincides locally with $\{ f \le f(p)\}$.

\emph{ (i) $\Rightarrow$ (ii):}  Suppose $Z$ fails to have property (ii), we will show that also property (i) cannot be satisfied by $Z$. Following the argument in \cite[Lemma 2.4]{White2016}, we can construct a function $f \in C^\infty(\Omega, \R)$ such that $f|_Z$ attains its maximum at a unique point $p \in Z$, i.e.
$$f(x)<f(p)\qquad \forall x \in Z,$$
$D f(p)\neq 0$, the super-level set $\{ x: f(x)\ge a \}$ is compact for every $a \in \R$ and 
\begin{equation}\label{1}
F_{ij}\left(p,\frac{D f(p)}{\abs{D f(p)}}\right) D_{ij}f(p)+  (\partial_{i} F_i)\left(p,\frac{D f(p)}{|D f(p)|}\right)|D f(p)| >  h \abs{Df}(p)
 \end{equation}
Up to translation, rotation and multiplication of $f$ by $\abs{D f(p)}^{-1}$, we can assume without loss of generality that $p=0$ and $D f(p)= e_{m+1}$. 

It is easy to verify that there exists an open neighborhood $U \ni p$ such that $\Sigma_0:=\{ x \colon f(x)=f(p)\}\cap U$ is a smooth sub-manifold of $\Omega$. Moreover, since $\Sigma_0$ is a level set of $f$, we know that
\begin{equation}\label{normale}
\nu_{\Sigma_0}(p)=D f(p)= e_{m+1},
\end{equation}
where $\nu_{\Sigma_0}(p)$ denotes the unit normal to $\Sigma_0$ at the point $p$.

 If we denote with $d(x)$ the signed distance function  from $\Sigma_0$
$$d(x):=\text{sign}(f(x)- f(p))\text{dist}(x,\Sigma_0),$$
since $\Sigma_0\cap U$ is smooth, there exists $r>0$ small enough such that $d$ is a smooth function on $B_r(p)$. Moreover $B_r(p)$ is contained in the $r$-neighborhood of $\Sigma_0$, since $p \in \Sigma_0$. 
Thanks to \eqref{normale}, we also deduce that
\begin{equation}\label{normale1}
Dd(p)= e_{m+1}.
\end{equation}
We observe that
$$B_r(p) \cap \{ d(x)>0\} \cap Z = \emptyset,$$
otherwise $f(p)$ would not be the maximum of $f|_Z$. We deduce that for every $\lambda >0$ the function 
$$g_\lambda(x):= (e^{\lambda d(x)} -1 )$$
satisfies $g_\lambda(x)\le 0$ for every $x \in Z\cap B_r(p)$. 
Fix a non negative cut off function $\varphi \in C^\infty_c(B_r(p))$ with $\varphi(x)=1$ on $B_{\frac{r}{2}}(p)$ and consider for every $\lambda>0$ the function 
\[ f_\lambda(x):= f(x) + \varphi(x) \lambda^{-\frac32} g_\lambda(x).\]
By the above considerations $f_\lambda$ restricted to $Z$ attains its maximum in $p$ and by direct calculations we have that for every $x \in B_{\frac{r}{2}}(p)$
\begin{align*}
D_if_\lambda(x)	&= D_i f(x) +  \lambda^{-\frac12} D_i d(x) e^{\lambda d(x)}\\
D_{ij} f_\lambda(x) &= D_{ij} f(x)  +  \lambda^{-\frac12} D_{ij} d(x) e^{\lambda d(x)} +  \lambda^{\frac12} D_id(x) D_jd(x)e^{\lambda d(x)}.  
\end{align*}
Evaluating the previous derivatives in $p$ and implementing \eqref{normale1}, we get
$$
D f_\lambda(p) = e_{m+1}+\lambda^{-\frac12} D_i d(p) e^{\lambda d(p)}=(1+ \lambda^{-\frac12}) e_{m+1}$$
and
\begin{equation*}
\begin{split}
D_{ij}f_{\lambda}(p) = D_{ij}f(p) + \lambda^{-\frac12} D_{ij} d(p) + \lambda^{\frac12} (e_{m+1} \otimes e_{m+1})_{ij}.
\end{split}
\end{equation*}
By homogeneity of $F$, we have $F_{m+1, m+1}(p, e_{m+1}) = 0 $, and combining the previous equation with \eqref{1}, we deduce that there exists $\lambda_0>0$ such that for all $\lambda > \lambda_0$
\begin{equation*}
 F_{ij}\left(p, e_{m+1}\right) D_{ij}f_\lambda(p) > h \abs{Df_\lambda}(p) - (\partial_{i} F_i)(p, D f_\lambda(p))\abs{Df_\lambda}(p)
\end{equation*}
We conclude that $f_\lambda$ fails the condition (i) for $\lambda$ chosen  sufficiently big, showing that
\[ \lim_{\lambda \to \infty} \inf_{v \in \mathbb S^m} F_{ij}(p,v) D_{ij}f_\lambda(p) = F_{ij}(p, e_{m+1}) D_{ij}f(p).\]
Indeed, 
for every $v \neq e_{m+1}$, the strict convexity of $F$ implies that $F_{m+1, m+1}(p, v) >0 $ and we can compute
\begin{equation*}
\begin{split}
 \lim_{\lambda \to \infty} F_{ij}(p,v) D_{ij}f_\lambda(p)&= \lim_{\lambda \to \infty}F_{ij}(p, v) D_{ij}f(p)\\
 &  \quad +\lim_{\lambda \to \infty} \lambda^{-\frac12} F_{ij}(p,  v) D_{ij} d(p) + \lambda^{\frac12} F_{m+1, m+1}(p, v) = + \infty,
\end{split}
\end{equation*}
unless \(v=e_{m+1}\).

\emph{ (ii) $\Rightarrow$ (i):} Suppose $Z$ fails to have property (i), we will show that this implies $Z$ does not satisfy property (ii). Similarly to the previous step, we can make use of the argument of \cite[Lemma 2.4]{White2016} and assume without loss of generality that $f \in C^\infty(\Omega, \R)$, $f|_Z$ attains its maximum at a unique point $p \in Z$ ($f(x)<f(p)$ for every $x \in Z$), the super-level set $\{ x: f(x)\ge a \}$ is compact for every $a \in \R$, there exist $r>0$ and $\delta>0$ small enough such that $f(x) < f(p)-\delta $ for all $x \not \in B_r(p)$ and 
\begin{equation*}
\inf_{v \in \mathbb S^m} F_{ij}(p,v) D_{ij}f(p)  >   h \abs{Df}(p) - (\partial_{i} F_i)\left(p, \frac{D f(p)}{|D f(p)|}\right)|D f(p)|,
 \end{equation*}
 where the right hand side is intended to be zero when $D f(p) = 0$.
 
If $\abs{D f}(p) \neq 0$, $Z$ fails to have property (ii) since trivially 
\[\inf_{v \in \mathbb S^m} F_{ij}(p, v) D_{ij}f(p)\le F_{ij}\left(p, \frac{D f(p)}{\abs{D f(p)}}\right) D_{ij}f(p).\]
Hence, we are reduced to consider the case $Df(p)=0$, i.e. the case in which there exists $v_0\in \mathbb S^m$ such that 
\begin{equation}\label{eq:caseDF=0}  
F_{ij}(p, v_0) D_{ij}f(p)=\inf_{v \in \mathbb S^m} F_{ij}(p,v) D_{ij}f(p) \ge \sigma >0.\end{equation}
This is done by relaxation. Up to a translation of $Z$ by $p$ and considering $f-f(p)$ we may assume without loss of generality that $p=0$ and $f(0)=0$. We can fix $M>0$ with $M \ge \sup\{\abs{f(x)}+ \abs{Df(x)}: x \in B_{2r}(0)\}$. Furthermore, for $\lambda >0$ we define the smooth auxiliary function
\[ g_{\lambda}(x,y):= f(y) - \lambda \abs{x-y}^4. \]
Observe that, by the stated properties of $f$, for every $x \in Z$ and every $y \notin B_r(0)$ we have $g_{\lambda}(x,y)\leq f(y)<-\delta < 0=g_\lambda(0,0)$. If $\abs{x-y}^4 > \frac{M}{\lambda}$, $y \in B_r(0)$ we have $g_{\lambda}(x,y) <0$. Hence for each $\lambda >0$ 
\[ m_{\lambda}:= \sup\{ g_\lambda(x,y) \colon x \in Z , y \in \Omega \} \]
is attained for a couple $(x_\lambda, y_\lambda) \in Z \times B_r(0)$ with $\abs{x_\lambda - y_\lambda}^4 \le \frac{M}{\lambda}$. \\
We moreover observe that $x_\lambda \to 0$ as $\lambda \to +\infty$. Indeed, for every $x\in Z \cap B_{r}(0)\setminus \{0\}$ and every $y \in B_{(\frac{M}{\lambda})^\frac14}(x)$, since $f(x)<0$ we get that for sufficiently large $\lambda$ 
\begin{equation*}
\begin{split}
g_\lambda(x,y) &\le f(y) = f(x) + (f(y)-f(x))\le f(x) + \sup_{z\in  B_{2r}(0)}|Df|(z) \left(\frac{M}{\lambda}\right)^\frac14 \\
&\le f(x) + M \left(\frac{M}{\lambda}\right)^\frac14 < 0=g_\lambda(0,0),
\end{split}
\end{equation*}
which implies that for $\lambda$ big enough $x$ is far enough from $x_\lambda$.\\
Since $y_\lambda \in B_{(\frac{M}{\lambda})^\frac14}(x_\lambda)$, then as $\lambda \to +\infty$ we get $x_\lambda - y_\lambda\to 0$ and consequently also $y_\lambda\to 0$.\\
For each couple $(x_\lambda, y_\lambda)$ we distinguish two cases:\\
\emph{First case: $x_\lambda = y_\lambda$.} Since $y \mapsto g_\lambda(x_\lambda, y)$ admits a global maximum in $y_\lambda$ we have $D_yg_\lambda(x_\lambda, y_\lambda)= Df(y_\lambda) =0$ and $D_{y}^2g_\lambda(x_\lambda, y_\lambda) = D^2f(y_\lambda) \le 0$. By convexity of $F$, it holds $F_{ij}(y,  v) \ge 0$ for every $(y,v)\in \Omega \times \mathbb S^m$, hence 
\[ F_{ij}(x_\lambda, v ) D_{ij}f(x_\lambda)=F_{ij}(y_\lambda, v ) D_{ij}f(y_\lambda) \le 0 \qquad \text{ for every } v \in \mathbb S^{m}. \]
Passing this inequality to the limit for $\lambda \to +\infty$ we get
$$F_{ij}(p, v ) D_{ij}f(p) \le 0 \qquad \text{ for every } v \in \mathbb S^{m}, $$
which contradicts \eqref{eq:caseDF=0}.\\
\emph{Second case: $x_\lambda \neq y_\lambda$.} As before $y \mapsto g_\lambda(x_\lambda, y)$ admits a global maximum in $y_\lambda$, hence 
\[ 0=D_yg_\lambda(x_\lambda, y_\lambda)= Df(y_\lambda) - 4\lambda \abs{y_\lambda - x_\lambda}^2 (y_\lambda - x_\lambda), \]
which gives in particular $Df(y_\lambda) \neq 0$. Furthermore 
$$\lim_{\lambda \to +\infty}\abs{Df(y_\lambda)} = 0,$$
 since $Df(0)=0$ and $y_\lambda \to 0$. 
Now consider the new function 
$$f_\lambda(x):=f(x+(y_\lambda - x_\lambda)).$$
 The function $f_\lambda|Z$ admits its maximum at $x_\lambda$ because for every $x \in Z$
\begin{equation*}
\begin{split}
f_\lambda(x)- \lambda \abs{y_\lambda - x_\lambda}^4&=f(x+(y_\lambda - x_\lambda)) - \lambda \abs{ x + (y_\lambda - x_\lambda)- x}^4\\
&=g_\lambda(x,x+(y_\lambda - x_\lambda))  \leq g_\lambda(x_\lambda,y_\lambda)\\
&= f(y_\lambda)- \lambda \abs{y_\lambda - x_\lambda}^4=f_\lambda(x_\lambda)- \lambda \abs{y_\lambda - x_\lambda}^4.
\end{split}
\end{equation*}
 Thanks to \eqref{eq:caseDF=0}, for $\lambda$ sufficiently large, we deduce that
\begin{align*}
	\inf_{v \in \mathbb S^m} F_{ij}(x_\lambda, v) D_{ij}f_\lambda(x_\lambda) = \inf_{v \in \mathbb S^m} F_{ij}(x_\lambda, v) D_{ij}f(y_\lambda) &> \frac{\sigma}{2} \\
	h \abs{Df_\lambda}(x_\lambda) - (\partial_{i} F_i)(p, D f_\lambda(x_\lambda))&< \frac{\sigma}{2}.
\end{align*}
We conclude that $Z$ fails to have property (ii).

\end{proof}

\begin{Lemma}\label{le1}
 Given $f\in C^\infty(\Omega)$ and $p \in \Omega$ such that $f(p)=0$ and $Df(p)\neq 0$, then there exists $N\subset  \Omega$ relatively closed with smooth boundary and $U\ni p$ open such that 
$$\{f\le 0 \} \subset N \quad \mbox{ and } \quad U \cap \{f\le 0 \} = U \cap N.$$
\end{Lemma}
\begin{proof}[Proof of Lemma \ref{le1}] If $0$ is a regular value of $f$, we can simply choose $N=\{ f \le 0 \}$. Otherwise we fix $r>0$ such that $\abs{Df(x)-Df(p)} \le \frac12 \abs{Df(p)}$ for all $x \in B_r(p)$.
We deduce that
\begin{equation}\label{conto in mezzo}
\abs{Df(x)}\ge  \abs{Df(p)} -\abs{Df(x)-Df(p)} \ge \frac12 \abs{Df(p)} \qquad \forall x \in B_r(p).
\end{equation}
Let $\phi \in C^\infty_c(B_r(p))$ be non negative, $\phi=1$ on $B_{\frac{r}{2}}(p)$ and $\abs{D\phi}< \frac{4}{r}$.  By Sard's theorem there is a regular value $c$ of $f$ with $0<\frac{4}{r}c < \frac{\abs{Df(p)}}{4}$.\\
We set
\[ \tilde{f}(x):= f(x) - \phi(x)c. \]
By the choice of $c$ and $\phi$ and thanks to \eqref{conto in mezzo}, we compute
$$\abs{D\tilde{f}(x)} \ge \abs{D{f}(x)}-c\abs{D\phi(x)}> \frac{\abs{Df(p)}}{2}-\frac{\abs{Df(p)}}{4}= \frac{\abs{Df(p)}}{4} \qquad \forall x \in B_r(p).$$
 Hence $0$ is a regular value of $\tilde{f}|B_r(p)$ and therefore $0$ is a regular value of $\tilde{f}$ on the whole set. Since $\tilde{f}=f$ on $U:=B_{\frac{r}{2}}(p)$, we infer that $U \cap \{f\le 0 \} = U\cap \{\tilde{f}\le 0 \}$ and we conclude that the relatively closed set $N:= \{ \tilde{f} \le 0 \}$ has the claimed properties. \end{proof}
\begin{Lemma}\label{le2}
 Given $N\subset \Omega$ relatively closed with smooth boundary and $p \in \partial N \cap \Omega$. There exists $f \in C^\infty(\Omega)$ and $U\ni p$ open such that 
$$N \subset \{f\le 0 \}\quad \mbox{ and } \quad U \cap \{f\le 0 \} = U \cap N.$$
\end{Lemma}
\begin{proof}[Proof of Lemma \ref{le2}]
 Fix a smooth proper function $u: \Omega \to \R$ with $u < 0 $ on $N$. We define the signed distance function $d$ defined as
 $$d(x):=\begin{cases}-\text{dist}(x,\partial N) & \text{if } x \in N\\
 \text{dist}(x,\partial N) & \text{if } x \notin N\end{cases}.$$
 Given $r>0$, as before we fix a non negative function $\phi \in C^\infty_c(B_r(p))$, with $\phi=1$ on $U:=B_{\frac{r}{2}}(p)$. It is now straightforward to check that, choosing $r$ small enough, the function 
\[ f(x):= \phi(x) d(x) + (1-\phi(x)) u(x) \]
has the claimed properties.
\end{proof}


\begin{Remark}\label{Rem:1} In Proposition \ref{prop:equivalence} above, we may replace $(ii)$ with the following equivalent condition:
\begin{itemize}
\item[(ii)'] If $P$ is a paraboloid  $P(x):= a_0 + \langle a_1 , x-p \rangle + \frac12 (x-p)^t A (x-p) $ for some $a_0 \in \R, a_1 \in \mathbb S^{m} $ and $A \in \R^{(m+1)\times (m+1)}$ and if $P|_Z$ has a local maximum at $p$, then 
\begin{equation}\label{paraboloid}
F_{ij}\left(p,a_1\right) A_{ij} + (\partial_{i} F_i)(p,a_1)\le h.
 \end{equation}
\end{itemize}
Indeed, the fact that (ii) implies (ii)' is immediate. For the converse, let $f$ as in (ii) and $p$ a local maximum of $f|_Z$. Consider for any $\varepsilon>0$  the paraboloid
\[ P_{\varepsilon}(x):= \left \langle \frac{D f(p)}{\abs{D f(p)}} , x - p \right \rangle + \frac{1}{2\abs{D f(p)}} D^2f(p)( (x-p)\otimes (x-p)) - \frac{\varepsilon}{2} \abs{x-p}^2. \]
Since  $f \in C^2$, for every $\varepsilon>0$ there exists $r_\varepsilon>0$ such that \[\sup_{x \in B_{r_\varepsilon}(p)} \frac{\abs{ \frac{f(x)}{\abs{Df(p)}}-P_{\varepsilon}(x)}}{\abs{x-p}^2} \le \frac{\varepsilon}{4}\,.\] 
Then $P_{\varepsilon}|_Z$ attains its local maximum in $p$.
Moreover we compute
$$DP_{\varepsilon}(p)= \frac{D f(p)}{\abs{D f(p)}} \qquad \text{ and } \qquad D^2P_{\varepsilon} = \frac{D^2f(p)}{\abs{D f(p)}} - \varepsilon \mathbf{1}.$$
 Letting $\varepsilon \to 0$ in \eqref{paraboloid}, we deduce the inequality in (ii) for $f$ in $p$.
\end{Remark}

The following is our main theorem. The proof is based on (the proof of) the maximum principle of Solomon and White for varifolds which are stationary with respect to an anisotropic integrand, see \cite{SolomonWhite}.

\begin{Theorem}\label{thm:area-blowupset}
	Let $\Omega \subset \R^{m+1}$ be open. Consider a sequence of varifold $(V_k)_k\subset \mathbb{V}_m(\Omega)$ and $h>0$ such that for every $K \subset \subset \Omega$ it holds
	\begin{equation}\label{eq:curvature bounded by h}
	\limsup_{k \to \infty} \sup \left\{  \delta_F V_k(X) - h \int \norm{X}_{F^*,x} \, F(x, \nu) \, dV_k(x,\nu) \colon \abs{X} \le \mathbf{1}_{K} \right\} < \infty.	
	\end{equation}
Then the area-blow up set
\[ Z:= \{ x \in \overline \Omega \colon \limsup_{k \to \infty} \|V_k\|(B_r(x))=+\infty \text{ for every $r>0$ } \}\]
is an $(m,h)$-set in $\Omega$ with respect to $F$. 
\end{Theorem}


\begin{proof}
We first observe that $Z$ is a closed set. Indeed, given $\{x_n\}_{n \in \N}\subset Z$, such that $x_n \to x \in \overline \Omega$, then, for every $r>0$, there exists $n$ big enough such that $B_{r/2}(x_n)\subset B_r(x)$. We deduce that
$$\limsup_{k \to \infty} \|V_k\|(B_r(x)) \geq \limsup_{k \to \infty} \|V_k\|(B_{r/2}(x_n))=+\infty,$$
which implies that $x \in Z$ and consequently that $Z$ is closed.

Assume now that $Z$ is not an $(m,h)$-set. Hence due to Proposition \ref{prop:equivalence} there is a smooth function $f:\Omega \to \R$ and a point $p \in \Omega\cap Z$ such that $f|_Z$ has a unique local maximum at $p$, $D f(p) \neq 0$ and (ii) fails. After translation by $p$ and rotation and scaling of $f$ we may assume that $p=0$, $f(p)=0$ and $D f(p)= -e_{m+1}$. The contradiction then reads 
\begin{equation}\label{stimaforte}
F_{ij}\left(p,\frac{Df(p)}{\abs{Df(p)}}\right) D_{ij}f(p)+ (\partial_{i} F_i)\left(p,\frac{Df(p)}{\abs{Df(p)}}\right)\abs{Df(p)}> h \abs{Df(p)}.
\end{equation}
Let us define the vector field
\begin{equation}\label{eq:good vectorfield}
	X(x)=X^i(x) e_i = F_i(x,Df(x)) e_i.
\end{equation}
Firstly note that $\langle X(x), Df(x)\rangle =  F(x,Df(x))$ hence $X$ is pushing along  ``outside''  the level sets $ \{ f \le t \}$. Furthermore 
\begin{equation}\label{eq:length of the good vectorfield}
	\norm{X}_{F^*,x} = 1. 
\end{equation}
Moreover, by \eqref{eq:good vectorfield identity -1}
\begin{equation}\label{eq:good vectorfield identity 1}
 -\langle H_F(x), X\rangle \, F(x, Df(x)) = B(x, Df(x)):DX + \langle D_1F(x,Df), X\rangle \end{equation}
where $H_F(x)$ is the $F$-mean curvature of a level set $\{ f= t\}$.

Now we want to show how this vector field can be used to derive the contradiction to \eqref{eq:curvature bounded by h}.
First fix a radius $r>0$ and $\delta>0$ such that
\begin{multline}\label{eq:stimaforte}
F_{ij}\left(x,\frac{Df(x)}{\abs{Df(x)}}\right) D_{ij}f(x)+ (\partial_{i} F_i)\left(x,\frac{Df(x)}{\abs{Df(x)}}\right)\abs{Df(x)}\ge (h+\delta) \abs{Df(x)}
\\
\text{ for all } x \in B_{2r}(0)
\end{multline}
and \begin{equation}\label{eq:bounds on Df}
 \frac{1}{2} \le \abs{Df(x)} \le 2 \text{ for all } x \in B_{2r}(0).	
\end{equation}
By \eqref{eq:good vectorfield}, we compute $\langle X, \frac{Df(x)}{\abs{Df(x)}} \rangle = F(x, \frac{Df(x)}{\abs{Df(x)}})$, which combined with \eqref{eq:stimaforte}, gives the following estimate on $B_{2r}(0)$ 
\begin{equation}\label{eq:stimaforte2}
F_{ij}\left(x,\frac{Df(x)}{\abs{Df(x)}}\right) D_{ij}f(x)+ (\partial_{i} F_i)\left(x,\frac{Df(x)}{\abs{Df(x)}}\right)\abs{Df(x)}\ge (h+\delta) F(x, Df(x))
\end{equation}
By assumption we have $Z \subset \{f \le 0\}$ and $Z \cap \{f=0\}=\{0\}$, hence there exists $\eta_1>0$ such that $f(x)<-\eta_1$ for all $x \in Z \setminus B_r(0)$. 
Now we fix a non-negative cut off function $\varphi(x)$ supported in $B_{2r}(0)$ with $\varphi(x) = 1$ on $B_{r}(0)$. For $0<\eta_2<\eta_1$ to be chosen later, we define the function 
\[ \eta(t):=\begin{cases} 0 &\text{ if } t \leq -\eta_2  \\ \eta_2+t &\text{ if } -\eta_2 \leq t \end{cases}.  \]
Now we consider the vector field
\begin{equation}\label{eq:good vector field 2}
Y(x) = -\varphi(x) \eta(f(x)) X.	
\end{equation}
Then we have 
\[ -DY = \varphi \eta \circ f DX + \varphi \eta'\circ f X \otimes Df + \eta\circ f X \otimes D\varphi. \]
Hence for every $a$ we have
\begin{align*}
- \delta_FV_k(Y)=& \int \varphi \eta\circ f \left( B(x,\nu):DX + \langle D_1 F(x,\nu), X\rangle \right) \\& + \varphi \eta'\circ f \left( B(x,\nu): X \otimes Df \right) + \eta \circ f \left(B(x,\nu) : X \otimes D\varphi \right)\, dV_k(x,\nu)\\
=& \int I + II + III \, dV_k(x,\nu). 	
\end{align*}
We analyze the three terms separately.
Note that $\abs{III} \le C \mathbf{1}_{B_{2r}\setminus B_r \cap \{ f \ge -\eta_1\}}$. Since by the choice of $r$ and $\eta_1$ we have $Z \cap B_{2r}\setminus B_r \cap \{ f \ge -\eta_1\} = \emptyset$ we have
\[ \abs{ \int III \, dV_k(x,\nu) } \le O(1) \text{ for all } k. \]
Concerning $II$ we have due the uniform convexity of $F$ there is a constant $c_F$
\begin{align*}
B(x,\nu):X \otimes Df(x) &= F(x,\nu) F(x,Df) - \langle D_2F(x,\nu), Df(x)\rangle \langle D_2F(x,Df(x)), \nu \rangle\\ &\ge c_F|Df(x)| \;{\dist}_{\mathbb{RP}^{m}}\Big(\frac{Df(x)}{|Df(x)|}, \nu\Big)^2\, F(x,\nu) 
\\
&= c_F|Df(x)|  \;d(x,\nu)^2 \, F(x,\nu),  	
\end{align*}
where, for \(v,w\in \mathbb S^m\), we set
\begin{equation}\label{e:prdist}
{\dist}_{\mathbb{RP}^{m}}(v,w):=\min\{|v+w|, |v-w|\},
\end{equation}
and we  introduced the function 
\begin{equation}\label{e:prdist2}
d(x,\nu):= \;{\dist}_{\mathbb{RP}^{m}}\Big(\frac{Df(x)}{|Df(x)|}, \nu\Big).
\end{equation}
We conclude taking into account \eqref{eq:bounds on Df} 
\[ \int II\, dV_k(x,\nu) \ge \frac12 c_F \int  \varphi \eta\circ f\, d(x,\nu)^2 \, F(x,\nu)\, dV_k(x, \nu).\]
It remains to estimate $I$. By \eqref{eq:stimaforte}, \eqref{eq:good vectorfield identity 1} and the \(C^2\) regularity of \(F\), there exists a constant $C_F \ge 0$ such that
\begin{align*}
 \abs{Df(x)}\Bigl( B(x,\nu)&:DX(x) + \langle D_1F(x,\nu), X\rangle \Bigr)
\\
&\ge B(x,Df(x)):DX(x) + \langle D_1F(x,Df(x)), X\rangle \\
&\quad- C_F |Df(x)|\;{\dist}_{\mathbb{RP}^{m}}\left(\frac{Df(x)}{\abs{Df(x)}}, \nu \right)
\\
&\ge  \abs{Df(x)} (h+ \delta) F(x, \nu)  - C_F |Df(x)| \;d(x,\nu)\, F(x,\nu).
\end{align*} 
Taking additionally into account that  $\{ \eta \ge \eta_2 \} \cap B_{2r} \cap Z = \emptyset$ and \eqref{eq:bounds on Df}, we conclude  
\begin{align*} \int I \, dV_k(x,\nu) \ge& (h+\delta) \int \varphi \eta\circ f \, F(x,\nu) \, dV_k(x,\nu)\\& - 2 C_F \int_{\{\eta<\eta_2\}} \varphi \eta \circ f \, d(x,\nu)\, F(x,\nu) \, dV_k(x,\nu)- O(1). \end{align*}
Combing all the estimates for $I- III$ we have

\begin{align*}
&\int I + II + III \, dV_k(x,\nu) - h\int \varphi \eta \circ f \, F(x,\nu)\, dV_k(x,\nu) \\
&\ge \int_{\{\eta < \eta_2\}} \varphi \left( \delta \,\eta\circ f - 2C_F \eta\circ f d(x,\nu) + \frac12 c_F \,\eta'\circ f d(x,\nu)^2\right) \, dV_k(x,\nu)- O(1). 
\end{align*}
Observe that $0\le \eta\circ f \le 2\eta_2$ on the set $\{ f< \eta_2\}$ and $\eta' =1$ on the set $\{ \eta > 0 \}$. Let us  consider the polynomial 
\[ 
p(\mu, t):= \delta\, \mu \, - 2C_F\, \mu t + \frac12 c_F \, t^2. 
\]
For a fixed $\mu\ge 0$ its minimum is obtained in $t_{\text{min.}}= \frac{2 C_F \mu}{ c_F}$ and takes the value
\[ 
p(\mu, t_{\text{min.}}) = \frac{\delta}{2}\, \mu - \frac{2 C_F^2 \,\mu^2}{c_F}. 
\]
Hence if $\mu \le 2\eta_2$ with $\eta_2>0$ sufficient small, $p(\mu,t)$ is non-negative i.e. for such a choice of $\eta_2$ we have
\begin{align*}
&\int I + II + III \, dV_k(x,\nu) - h\int \varphi \eta \circ f \, F(x,\nu)\, dV_k(x,\nu) \\
&\ge \int_{\{\eta < \eta_2\}} \varphi  \frac{\delta}{2} \,\eta\circ f  + \int_{\{\eta < \eta_2\}} \varphi p(\eta\circ f, d(x,\nu)) dV_k(x,\nu) - O(1) \\
&\ge \int_{\{\eta < \eta_2\}} \varphi  \frac{\delta}{2} \,\eta\circ f \, d\norm{V}_k(x) - O(1). 
\end{align*}
Since $B_{\frac{r}{2}} \cap \{ \eta \circ f < \eta_2 \}$ 
is an open neighbourhood of $0$ and $0 \in Z$, we conclude that  
\[ \lim_{k \to \infty} \int_{\{\eta < \eta_2\}} \varphi  \frac{\delta}{2} \,\eta\circ f \, d\norm{V}_k(x) = +\infty, \]
contradicting the assumption \eqref{eq:curvature bounded by h} and proving the theorem. 
\end{proof}

\subsection{Consequences of Theorem \ref{thm:area-blowupset}}
By repeating the arguments of \cite{White2016}, we can now derive several properties of area blow-up sets (and more in general of \((m,h)\)-sets).

\begin{Proposition}\label{prop:closedness}
Let $\Omega\subset \R^{m+1}$ be open, $(F_k)_k$ be a sequence of anisotropic integrands, and $(Z_k)_k$ be a sequence of $(m, h_k)$-subset of $\Omega$ with respect to the integrand $F_k$. Suppose that $F_k$ converges uniformly on compact subsets of $\Omega$ to some integrand $F$, $Z_k$ converges in Hausdorff distance to a closed set $Z$ and $h_k \to h$, then $Z$ is an $(m,h)$-subset of $\Omega$ with respect to the integrand $F$. 
\end{Proposition}

\begin{proof} We will prove that the condition (ii)' in Remark \ref{Rem:1} holds. Let 
$$P(x)=a_0 + \langle a_1, x \rangle + \frac12 x^t A x \qquad \text{for some }  a_0 \in \R, a_1 \in \mathbb S^{m} \text{ and } A \in \R^{(m+1)\times (m+1)}$$
 be a paraboloid that realizes its maximum on $Z$ in $p\in \Omega$. Let $r >0$ such that $B_{r}(p)\subset \subset \Omega$. For any $\varepsilon >0$ and $k$ sufficient large, the map
 $$P_\varepsilon(x):=P(x)-\varepsilon \frac{\abs{x-p}^2}2$$
  realizes a strict local maximum on $Z_k\cap B_{r}(p)$ along a sequence of point $p_k \in Z_k\cap B_{r}(p)$, such that $p_k \to p$.

 Since $Z_k$ are $(m, h_k)$-subset of $\Omega$, we can apply the characterization (ii)' in Remark \ref{Rem:1} to $P_\varepsilon$ to deduce that 
\[ F_{ij}\left(p_k,a_1\right) (A_{ij}-\varepsilon \delta_{ij}) \le h_k  - (\partial_{i} F_i)(p_k,a_1) +C |p_k-p|.\]
Passing to the limit as $k\to \infty$ and \(\varepsilon \to 0\), we obtain
\[ F_{ij}\left(p,a_1\right) A_{ij} \le h - (\partial_{i} F_i)(p,a_1).
\]
\end{proof}

\begin{Corollary}\label{cor:Z-blow-up}
Let $\Omega\subset \R^{m+1}$ be open and $Z\subset \Omega$ be an $(m,h)$-set with respect to the anisotropic integrand $F$. Consider a sequence $r_k \searrow 0$ and a point $p \in \Omega \cap Z$ such that
\[ Z_i:= \frac{ Z- p}{r_i} \to Z_\infty \qquad \mbox{in Hausdorff distance}. \]
Then $Z_\infty$ is an $(m,0)$-set of $\R^{m+1}$ with respect to the frozen integrand $F_p(\nu):=F(p, \nu)$.
\end{Corollary}
\begin{proof}
It is straight forward to check that for every $r>0$ and $q \in \Omega$ 
\[ \frac{Z-q}{r} \]
is an $(m, rh)$-set with respect to the integrand
\[ F_{q,r}(x,\nu):= F(q+rx, \nu). \]
By Proposition \ref{prop:closedness}, $Z_\infty$ is an $(m,0)$-subset of the integrand
\[ F_p(\nu)= \lim_{k \to \infty} F_{p, r_k}(x, \nu). \]
\end{proof}

A further consequence of Theorem \ref{thm:area-blowupset} is a constancy property, compare with \cite[Section 4]{White2016}:
\begin{Proposition}
Let $\Omega \subset \R^{m+1}$ be open and $Z$ be an $(m,h)$-subset of $\Omega$ with respect to an anisotropic integrand $F$. Suppose $Z$ is a subset of a connected, $m$-dimensional, properly embedded $C^1$-submanifold $M$ of $\Omega$. Then
\[ \text{either} \quad Z =\emptyset \quad \text{ or } \quad Z = M. \]
\end{Proposition}

\begin{proof} 
If $Z =\emptyset$ there is nothing to prove. Assume that $Z\neq \emptyset$ and suppose by contradiction that $Z\neq M$. Since $Z$ is closed, there exists $B_r(q) \subset \Omega \setminus Z$ with $q \in M$ and $p \in Z \cap \overline{B_r(q)}$. For a sequence of positive numbers $\lambda_k \searrow 0 $ consider 
\[ Z_k := \frac{ Z - p}{\lambda_k} \quad \text{ and } \quad M_k:= \frac{ M - p}{\lambda_k}. \]
Due to the regularity of $M$, we have that $M_k \setminus B_{\frac{r}{\lambda_k}}(\frac{ q - p}{\lambda_k})$ converges in Hausdorff distance to a half plane $H$ of $T_pM$. Hence, passing to a subsequence, $Z_k \to Z_\infty$ in Hausdorff distance, with $Z_\infty \subset H$ and $0 \in Z_\infty$. 
After a rotation $O$, we may assume that $H=\{ x \in \R^{m+1} \colon x_{m+1} =0, x_1\ge 0 \}$. By corollary \ref{cor:Z-blow-up} we have that $Z_\infty$ is an $(m,0)$-subset of $\R^{m+1}$ with respect to the frozen integrand $\hat{F}(\nu) := F(p,O\nu)$.
Now consider the function 
$$f(x):= - x_1 + x_1^2 + x_{m+1}^2.$$
 Observe that $f$ takes a strict local maximum at $0$ on $H$, hence $f|_{Z_\infty}$ has a strict local maximum in $0$, but this contradicts the characterization (ii) of Proposition \ref{prop:equivalence},  since 
\[ D^2\hat{F}(e_1) (e_1 \otimes e_1 + e_{m+1} \otimes e_{m+1} ) >0. \]
\end{proof}

For the sake of completeness we prove also the anisotropic counterpart of the ``classical'' constancy theorem for varifolds. The reader may compare it with \cite[Theorem 8.4.1]{Simon} for the proof in the isotropic setting.
\begin{Proposition}\label{prop:constancy classical} Given $V\in \mathbb V_m(\Omega)$ wich is stationary with respect to an anisotropic integrand $F$. Let $\spt(V) \subset M$, where $M$ is a connected $M$-dimensional $C^2$ submanifold of $\Omega$, then $V= \theta_0 \, \mathcal{H}^m\res M \otimes \delta_{T_xM}$.
\end{Proposition}
\begin{proof}
The strategy of the proof is similar to the one for the area functional, compare \cite[Theorem 8.4.]{Simon}. To simplify the presentation, we divide the proof in two steps:
\begin{itemize}
\item[Step 1)] if $M$ is a plane, i.e. $M = \{ x_{m+1} = 0 \}$, and $\Omega = B_{2r}(0)$, then the conclusion of the proposition holds on $B_r(0)$.
\item[Step 2)] we reduce the general case to the case in Step 1.
\end{itemize}
\emph{Proof of Step 1:} We will write $x=(y, z)\in \R^m \times \R$ for the coordinates in $\R^{m+1}$ i.e. $M = \{ z=0\}$. Consider the vectorfield 
\[ X(x):= \varphi(y) \eta(z) f(z) D_2 F(x, e_{m+1}) \]
where $\varphi \in C^1_c(B_r^m(0))$, $f, \eta \in C^1(\R)$ satisfying $f(0)=0, f'(0) \neq 0$ and $\eta$ non-negative with $\eta(z)=0$ for $\abs{z} > r$ and $\eta(z)=1$ for $\abs{z} < \frac{r}{2}$. \\
Since $\spt(V) \subset M$, $f =0 $ on $M$, $\eta=1$ on $M$ and $\eta'=0$ on $M$, the first variation formula (see \cite[section 5]{DePhilippisDeRosaGhiraldin}) reduces to 
\[ 0 = \delta_FV(X) = \int B_F(x, \nu) : (\varphi(y) f'(0) D_2 F(x, e_{m+1}) \otimes e_{m+1}) \, dV(x, \nu). \]
Since $f'(0) \neq 0$, the previous equation implies that
 \[B_F(x, \nu) : D_2 F(x, e_{m+1}) \otimes e_{m+1} 
 =0, \qquad \text{for $V$-a.e. }(x,\nu),\] which, by strict convexity of $F$, is only possible when $\nu = \pm e_{m+1}$ for all $x \in B_r(0)\cap \spt(V)$. This shows that the tangent space of $V$ agrees with the tangent space of $M$, that is
 $$V=\|V\|\otimes \left(\frac 12\delta_{e_{m+1}} + \frac 12 \delta_{-e_{m+1}}\right).$$
 Furthermore, we consider the vectorfield 
\[ X(x):= \varphi(y)\eta (z) e_i, \qquad \text{ for every} 1\le i \le m. \]
Since $\eta = 1 $ on $M$ and $B_F(x, \nu)$ is even in the second variable, the first variation formula reads
\begin{align*}
 0 = \delta_F V(X) &= \int B_F(x, e_{m+1}) :  (e_i \otimes D \varphi)  + \partial_i F(x,e_{m+1}) \varphi \, d\|V\|(x) \\
 &= \int F(x, e_{m+1}) \partial_i\varphi + \partial_i F(x,e_{m+1}) \varphi \, d\norm{V} (x)\\&= \int \partial_i (F(x,e_{m+1}) \varphi) \, d\norm{V}(x).
\end{align*}
Hence $\norm{V}$ is constant on $M\cap B_r(0)$. This concludes the proof of step 1. \\

\emph{Proof of Step 2:} Fix any $p \in M\cap \spt(V)$ and $0< r< \dist(x, \partial \Omega)$  such that the following holds: there is a $C^2$ function $\Phi: B_{2r}(p) \to B_{2r}(0)\subset \R^{m+1}$ with $\Phi(M\cap B_{2r}(0)) = \{ x_{m+1} = 0 \}\cap B_{2r}(p)$.  We replace $V$, $M$ and $F$ in $\Omega$ respectively with $V':=\Phi^\# V$, $M':=\Phi(M)$ and $F':={\Phi^{-1}}^\# F$ in $B_{2r}(0)$. By construction $V', M', F'$ are all as in Step 1. Hence we deduce that in $B_r(0)$ $$V' = \theta_0  \, \mathcal{H}^m\res M'\otimes \left(\frac 12\delta_{\nu_x} + \frac 12 \delta_{-\nu_x}\right), \quad \text{where $\nu_x$ is the normal vectorfield to $M$}.$$ But this implies that $V = \theta_0  \, \mathcal{H}^m\res M\otimes \left(\frac 12\delta_{\nu_x} + \frac 12 \delta_{-\nu_x}\right)$ in $B_r(p)$ and the proposition follows. 
\end{proof}

\section{Boundary curvature estimates}\label{sec:bound}
In this section we prove the following theorem which easily implies Theorem \ref{thm:boundaryintro}. Recall that a set \(\Omega\) is strictly \(F\)-convex in  \(B_{R}\) if
\[
H_F(x,\partial \Omega)\ge c>0\qquad\text{for all \(x\in \Omega\cap B_R\)}.
\]
It easily follows by \eqref{calcolo} that a uniformly convex set is strictly \(F\)-convex in in sufficiently small balls.

\begin{Theorem}\label{thm:curvature estimates at the boundary} Let $\Omega \subset \R^{3}$ s.t. $\partial \Omega \cap \overline{B_{2R}}$ is $C^{3}$ and \(\Omega\) is strictly \(F\)-convex in  \(B_{2R}\).  Let $\Gamma$ be a $C^{2,\alpha}$ embedded curve in $\partial \Omega \cap B_{2R}$  with $\partial \Gamma \cap  B_{2R} = \emptyset$. Furthermore let $M$ be an $F$-stable, $C^2$ regular surface in $\Omega$ such that $\partial M \cap B_R = \Gamma$. Then there exists a constant $C>0$ and a radius $r_1>0$ depending only on $F, \Omega, \Gamma$  such that 
\begin{equation*}
 \sup_{\substack{p \in B_{\frac{R}{2}} \cap \Omega \\ \dist(p, \Gamma)< r_1}} r_1 |A(p)| \le C.
\end{equation*}
Moreover the constants \(C\) and \(r_1\)  are uniform as long as \(\Omega\), \(\Gamma \) and \(F\) vary in compact classes\footnote{ For a family of curves  \(\Gamma_{\alpha}\) this amounts also in asking that all the considered curves should be ``uniformly'' embedded:
\[
\inf_{\alpha} \inf_{\substack{x\ne y\\ \,x,y\in \Gamma_\alpha}}\frac{ \dist_{\Gamma}(x,y)}{|x-y|}>0.
\]
}.
\end{Theorem}

We start with the following simple lemma.

\begin{Lemma}\label{lem:nonemptyblowup}
 Let $\{\mu_j\}_{j\in \N}\subset \mathcal M_+(\R^{m+1})$ be a sequence of Radon measures such that
 $$\lim_{j\to \infty} \mu_j(\overline{B}_1)= + \infty.$$
   Then the ``area-blow up set''
\[ Z:= \{x \in \R^{m+1} \colon \limsup_{j \to \infty} \mu_j(B_r(x)) = + \infty \text{ for every } r >0 \} \]
satisfies $Z \cap \overline{B}_1 \neq  \emptyset $. \end{Lemma}

\begin{proof} Up to consider as new sequence of measures $\mu_j|\overline{B_1}$, we can assume that $\spt(\mu_j) \subset \overline{B}_1$. We claim that there exists a sequence of cubes $\{C_i\}_{i\in \N}$ with side length $l_i$ such that 
\begin{itemize}
\item[(i)] $C_{i+1} \subset C_i$ for all $i \in \N$;
\item[(ii)] $l_i=2^{1-i}$;
\item[(iii)]  $\limsup_{j \to \infty} \mu_j(C_i)=+\infty$ for all $i \in \N$.
\end{itemize}
%
%
%
%
We will prove this claim by induction on $i$. We remark that $C_0$ exists: it is enough to consider a cube containing $\overline{B}_1$, for instance $\overline{B}_1 \subset [-1,1]^m=:C_0$, so that we have 
\[\limsup_{j \to \infty} \mu_j(C_0) = + \infty. \]
\emph{Proof of the Inductive step:}
Let $\mathcal{C}$ the collection of the dyadic cubes that are obtained by dividing $C_i$ into $2^{m+1}$ sub-cubes with half side length. Suppose 
$$\limsup_{j\to \infty} \mu_j(C')< \infty  \qquad \forall C'\in \mathcal{C}.$$
 Since there are only $2^{m+1}$ of these cubes, there exists $j_0 \in \N$ and $K>0$ such that 
\[ \mu_j(C') \le K \qquad \forall j \ge j_0, \quad \forall C' \in \mathcal{C}. \]
But this contradicts the assumption, since
\[ \mu_j(C_i) \le \sum_{C' \in \mathcal{C}} \mu_j(C') \le 2^{m+1} K \qquad \forall j \ge j_0.\]
We consequently can find a cube $C_{i+1} \subset C_i$ satisfying the properties (i), (ii) and (iii). \\
As a consequence we obtain a decreasing sequence of dyadic closed cubes $\{C_i\}_{i=0}^\infty$ with nonempty intersection, i.e. there exists $x \in \bigcap_{k=0}^\infty C_i$. \\
Since for every $r>0$ there exists $i \in \N$ such that $C_i \subset B_r(x)$, we have 
$$\limsup_{j \to \infty} \mu_j(B_r(x))=+\infty.$$
 This implies that $x$ is in the area blow up set. Finally since $x$ must be in the support of infinitely many $\mu_j$, we have $x \in \overline{B}_1$. This concludes the proof of this lemma.\end{proof}

The next proposition ensures that we have a local bound on the mass ratio, indeed assuming the contrary the varifolds associated with
\[
M_{x,r}=\frac{M-x}{r}
\]
would have unbounded masses. If \(Z\) is the area blow up set for this sequence, we can exploit our \(F\) convexity assumption together with the Hopf lemma to  show that \(Z\) is contained in a wedge, this contradicts the fact that it is an \((m,h)\)-set.

\begin{Proposition}\label{prop:no-boundary blow up}
Let $\Omega \subset \R^{m+1}$ such that $\partial \Omega \cap \overline{B_{2R}}$ is $C^{3}$ and \(\partial\Omega\) is strictly $F$ convex in \(B_{2R}\).  Let $\Gamma$ be a $C^{2,\alpha}$ embedded $(m-1)$-submanifold in $\partial \Omega \cap B_{2R}$  with $\partial \Gamma  \cap B_{2R} = \emptyset$. Furthermore, let $M$ be a $C^2$ stationary (i.e. $\hat{\delta}_F M=0$) manifold in $\Omega$ such that $\partial M \cap B_R = \Gamma$. Then there exists a constant $C$ and a radius $r_0>0$ depending only on $F, \Omega, \Gamma$  such that
\begin{equation}\label{eq:blowup} \sup_{\substack{q \in \Gamma \cap B_{R}\\ r< r_0}} \frac{ \H^m(M\cap B_r(q))}{w_m r^m} \le C < \infty. \end{equation}
\end{Proposition}

\begin{proof} We split the proof in two steps:

\emph{Step 1:} Proposition \ref{prop:no-boundary blow up} holds under the following additional Assumption \ref{ass.smallball}:
\begin{Assumption}\label{ass.smallball} There exists $0<\delta<\frac{1}{4}$ such that:
\begin{enumerate}
\item $\Omega \cap B_2 = \{ x_{m+1} \ge \Phi(x_1, \dotsc, x_m) \} \cap B_2$ for some $\Phi \in C^{2,\alpha}(\R^m, \R)$.  Furthermore we have 
$$\Phi(0)=0, \quad D \Phi(0)=0 \quad \text{(i.e. $T_0\partial \Omega = e_{m+1}^\perp$) and } \quad\norm{\Phi}_{C^{2,\alpha}}<\delta;$$
\item Let $\mathcal{F}(x,y,p)$ denote the non-parametric function associated to $F$
\[ \mathcal{F}(x,y,p):= F\left ((x,y),  p_1e_1+\dots+p_me_m - e_{m+1}\right) \]
and $L$ be the Euler-Lagrange operator for $\mathcal{F}$. Then, for every $U \subset B^m_2$ with smooth boundary, $f \in C^{0,\alpha}$ and $g \in C^{2,\alpha}$ with
$$\norm{f}_{C^{0,\alpha}}< \delta, \quad \norm{g}_{C^{2,\alpha}} < \delta,$$
the boundary value problem 
\[ \left\{\begin{aligned}
L u &= f &&\text{ in } U\\
u &= g &&\text{ on } \partial U
\end{aligned}\right.\]
has a unique solution $u \in C^{2,\alpha}(U,\R)$ such that
\[
 \norm{u}_{C^{2,\alpha}(U,\R)} \le C (\norm{f}_{C^{0,\alpha}} + \norm{g}_{C^{2,\alpha}} );
\]
\item for all $x \in \partial \Omega\cap B_2$ we have 
 \[
  0< h_{min} < L(\Phi) < h_{max} < \delta. 
\] 
Note that
\[
F(x,\nu(x)) H_F(x) = \frac{L(\Phi)(x)}{\langle \nu(x), e_{m+1}\rangle} \nu(x)
\]
 for all $x \in \partial \Omega \cap B_2$, where $\nu(x)$ is the normal of $\partial \Omega$ at the point $x$. 

\item $\Gamma \subset \partial \Omega$ is $C^{2,\alpha}$. 
\end{enumerate}
\end{Assumption}

\emph{Step 2:} There exists a radius $0<R_0\le R$ such that for every $p \in \partial \Omega$ the rescaled domain $\frac{\Omega-p}{R_0}$ and the rescaled manifold $\frac{M-p}{R_0}$ satisfy the conditions of Assumption \ref{ass.smallball}.\\

By a classical covering argument, one can show that Step 1 and Step 2 together imply Proposition \ref{prop:no-boundary blow up}.

\emph{Proof of Step 1:} Assume the conclusion \ref{eq:blowup} does not hold in $B_1$, then there exists a sequence $M_k, r_k, p_k$ satisfying
\begin{enumerate}
\item $\Gamma_k :=\partial M_k \subset \partial \Omega$ with uniformly bounded $C^{2,\alpha}$-norm;
\item $p_k \in \Gamma_k \cap B_1$, $0<r_k< \frac{1}{k}$ and 
\begin{equation}\label{piublowupdicosisimuore}
\frac{\H^m(M_k\cap B_{r_k}(p_k))}{r_k^m} > k. 
\end{equation}
\end{enumerate}
We denote with $\gamma_k$ the projection of $\Gamma_k$ onto the plane $\{ x_{m+1} = 0 \}$, i.e.
$$\Gamma_k = \mathbf{G}_{\Phi}(\gamma_k),$$ 
where $\mathbf{G}_\Phi(x):=(x, \Phi(x))$ is the graph map of $\Phi$. Up to subsequences, and performing if necessary a rotation of $B_2$, we may assume that
\begin{enumerate}\setcounter{enumi}{2}
\item there exists $x_0 \in \overline{B}_1$ such that $p_k=(x_k,\Phi(x_k)) \to p_0=(x_0,\Phi(x_0))$;
\item $\hat{\nu}_k(x_k) \to e_{m}$, where $\hat{\nu}_k(x)$ denotes the normal of $\gamma_k$ in the plane $\{x_{m+1} =0 \}$ at the point $x \in \gamma_k$.
\end{enumerate}
To set up the contradiction we need the following additional construction:\\
Consider  $r_0>0$ small enough so that for all $k \in \N$ we have 
$$r_0 < \min\left\{\frac{1}{\norm{ \mathbf{A}_{\gamma_k} }_{\infty}},\frac12\right\},$$
 where $\mathbf{A}_{\gamma_k}$ denotes the second fundamental form of $\gamma_k$. This is possible since we assumed that the $C^{2,\alpha}$-norm of $\Gamma_k$ is uniformly bounded. \\
For every $k \in \N$ we define the pair of balls
\[ B_k^{\pm} := B^m_{r_0}(x_k \pm r_0\hat{\nu}_k(x_k)) \subset \{x_{m+1} =0 \}. \]
By the choice of $r_0$, we have ensured that $\overline{B_k^\pm} \cap  \gamma_k = \{x_k\}$.
 For each $0\le s \le \delta$, using Assumption \ref{ass.smallball} (2), let $u^\pm_{k,s} \in C^{2,\alpha}(B^\pm_k(x))$ be the unique solution to the boundary value problem 
\[ \left\{\begin{aligned}
L u^\pm_{k,s} &= s &&\text{ in } B_k^\pm \\
u_{k,s}^\pm &= \Phi &&\text{ on } \partial B_k^\pm \, .
\end{aligned}\right.\]
Observe that, by the classical Hopf-maximum principle, if $s> h_{max}$ we have $u^\pm_{k,s} < \Phi$ and if $s< h_{min}$ then $u^\pm_{k,s} > \Phi$. We claim that the graphs of $u^\pm_{k,s}$ never touch $M_k$ in the interior of the cylinders $B^\pm_k \times \R$ for $s \ge \frac12 h_{min}$. Indeed, for $s> h_{max}$ this is obvious since $M_k \subset \Omega$. Suppose there is a first $\frac12 h_{min} < s\le h_{max}$ where for instance the graph $u^+_{k,s}$ touches $M_k$ at a point $q=(y,u^+_{k,s}(y))$. Then $T_qM_k = (- D u^+_{k,s}(y), 1)^\perp$ and $M_k$ is locally the graph over the plane $\{x_{m+1} = 0 \}$ around $y$ by a map $f_k$. Since $M_k$ is stationary we have $L(f_k)=0$, but this contradicts the  strong maximum principle.

For $k \in \N$, define 
\[
u_k^\pm:=u^\pm_{k,\frac12 h_{min}}.
\]
 By the Hopf boundary point lemma we can compare $\Phi$ with $u_k^\pm$ at $x_k$, obtaining the existence of $c_H>0$ depending only on  $F$ and $\partial \Omega$ such that  
\begin{equation}\label{eq:Hopf} \min\left\{ \frac{\partial u_k^+(x_k)}{\partial \hat{\nu}_k(x_k)} - \frac{\partial \Phi(x_k)}{\partial \hat{\nu}_k(x_k)}, -\frac{\partial u_k^-(x_k)}{\partial \hat{\nu}_k(x_k)} + \frac{\partial \Phi(x_k)}{\partial \hat{\nu}_k(x_k)} \right\} > c_H. \end{equation}
Furthermore by (2) in Assumption \ref{ass.smallball}, $\norm{u^\pm_k}_{C^{2,\alpha}}$ is uniformly bounded on $B_k^\pm$.\\
Now we consider the blow-up sequence
\begin{itemize}
\item $M_k':= \frac{M_k - p_k}{r_k}$ in $\Omega'_k:= \frac{\Omega_k - p_k}{r_k}$;
\item $\Gamma_k' = \partial M_k'= \frac{\partial M_k - p_k}{r_k}$ projecting to $\gamma'_k = \frac{\gamma_k - x_k}{r_k}$ in $\{x_{m+1}=0\}$;
\item $d_k^\pm(y)= \frac{u^\pm_k(x_k + r_k y) - \Phi(x_k+r_k y)}{r_k}$ on $\frac{1}{r_k}(B^\pm_k - x_k)$.
\end{itemize}
Observe that, by the regularity assumption on $\Omega$ and $\Gamma_k$ and the estimates on $u_k^\pm$, we have (up to a subsequence)
\begin{itemize}
\item[(i)] $\partial \Omega'_k \to T_{p_0}\partial \Omega$, i.e. $\Omega'_k \to \{ x_{m+1} \ge \langle D\Phi(x_0), x \rangle\}$;
\item[(ii)] $\gamma_k' \to \{ x_m =0 \}$;
\item[(iii)] $d^\pm_k(y) \to a^\pm y_m$ for $y \in \R^m \cap \{ \pm y_m \ge 0 \}$ with $a^+, -a^->c_H$
\end{itemize}
Indeed (ii) follows by property (4). Point (iii) is a consequence of the fact that $\frac{1}{r_k}(B^\pm_k - x_k) \to \R^m \cap \{ \pm y_m \ge 0 \}$ and that, by construction, we have $d^\pm_k=0$ on $\partial \frac{1}{r_k}(B^\pm_k - x_k)$. The last part of (iii) is a consequence of \eqref{eq:Hopf}.

By \eqref{piublowupdicosisimuore} and the definition of $M'_k$, we observe that the sequence of Radon measures $\mu_k:= \mathcal{H}^{m}\res M'_k$,  satisfies the assumptions of Lemma \ref{lem:nonemptyblowup}, hence  $Z\cap \overline{B}_1 \neq \emptyset$ where \(Z\) is the area blow up set for  \(M'_k\)

Since $M$ is a stationary manifold (i.e. $\hat{\delta}_F M=0$), by \eqref{eq:boundaryvariation} we can estimate for every vectorfield $X$ with $\abs{X} \le \mathbf{1}_{B_R}$ 
\[| \delta_F M_k'(X)| \leq \int_{\Gamma'_k} |X|  \le \mathcal{H}^{m-1}(\Gamma'_k\cap B_R). \]
Applying Theorem \ref{thm:area-blowupset}, we get that $Z$ is an $(m,0)$-set in $\R^{m+1}$ for the frozen integrand $F_{p_0}:\nu \mapsto F(p_0, \nu)$. 

Moroever, combining (i) and (iii), we know that 
\begin{equation}\label{eq:subset}
 Z \subset \{ (x,x_{m+1}) \colon x_{m+1} \ge \langle D\Phi(x_0), x \rangle + c_H \abs{x_m} \}. 
 \end{equation}
 
We will show that this contradicts the fact that $Z$ satisfies the characterization  (ii) in Proposition \ref{prop:equivalence} for being an $(m,0)$-set for an appropriate choice of a function $f$. We can assume $c_H \le \frac{1}{4} $ (up to replace $c_H$ with $\min(c_H, \frac14)$). We set 
$$T:=4 \frac{ 1 +\abs{D \Phi(x_0)}}{c_H}$$ and consider $\varepsilon >0$ to be chosen later. We define the function 
\[ f(x,x_{m+1}):= -x_{m+1} + \langle D\Phi(x_0), x \rangle + \frac{c_H}{2T} \left( x^2_m- \varepsilon x^2 \right).\]
On $\{ (x,x_{m+1}) \colon x_{m+1} \ge \langle D\Phi(x_0), x \rangle + c_H \abs{x_m} \}\cap \{ x_m = T \} $ we have 
\begin{align}\label{bordo striscia}
f(x,x_{m+1}) &= -x_{m+1} +\langle D\Phi(x_0), x \rangle + c_H \abs{x_m} + c_H \left( \frac{x_m^2}{2T}  - \abs{x_m} - \frac{\varepsilon x^2}{2T}  \right) \nonumber \\
&\leq  0 + c_H \left( \frac{x_m^2}{2T}  - \abs{x_m}\right) = - c_H\frac{T}{2}  \le -2 (1 + \abs{D \Phi(x_0)}).
\end{align}
But for every $x \in \overline{B}_1$ and choosing $\varepsilon$ sufficiently small, we have 
\[ f(x) > - \frac{3}{2} (1 +\abs{D \Phi(x_0)}). \]
Combining the previous inequality with \eqref{eq:subset} and \eqref{bordo striscia}, we deduce that $f|_Z$ takes a local maximum at some point $p=(\hat{x},\hat{x}_{m+1})$ with $\abs{\hat{x}_m} < T$. Now we claim that this contradicts (ii) in Proposition \ref{prop:equivalence} for sufficient small $\varepsilon>0$.
Indeed we can compute
\begin{align*}
D f(x, x_{m+1}) &=  - e_{m+1} + D \Phi(x_0) + \frac{c_H}{T}( x_m e_m - \varepsilon x ),\\
D^2f(x,x_{m+1}) &= \frac{c_H}{T}( e_m \otimes e_m - \varepsilon \Id_{\R^{(m+1)\times(m+1)}} );
\end{align*}
where $\Id_{\R^{(m+1)\times(m+1)}}$ is the $(m+1)$-dimensional identity matrix. Observe that there exits $\Lambda >0$ such that $\trace(D^2F_{p_0}(\nu))< \Lambda$ for all $\nu \in \mathbb S^{m}$. Furthermore for every $0<\eta <1$ there exists some $\lambda>0$ such that 
\[ D^2F_{p_0}(\nu): e_m \otimes e_m >\lambda \qquad  \text{ for all } \nu \in \mathbb S^m \text{ verifying } \abs{\langle \nu , e_m \rangle }< 1-\eta.\]
For every $x$ such that $\abs{x_m}\le T$, by Assumption \ref{ass.smallball} (1), we can compute
\[ \abs{\langle D f(x, x_m) , e_m\rangle} = \abs{\partial_m\Phi(x_0) + \frac{c_H}{T}(1-\varepsilon) x_m} \leq \abs{\partial_m\Phi(x_0)} + \frac{c_H}{T} \abs{x_m} \le \frac{1}{4}+c_H \leq \frac{1}{2}. \]

Since $\abs{D f(x,x_m)}\ge \langle Df(x,x_m),-e_{m+1}\rangle \ge 1$, we deduce that for every $x$ with $\abs{x_m}\le T$ 
$$\abs{ \left \langle \frac{D f(x)}{\abs{D f(x)}} , e_m \right \rangle } \le \frac12.$$
 If we choose $\varepsilon$ sufficiently small we compute in the local maximum point $p=(\hat{x},\hat{x}_{m+1})$
\[ \left \langle D^2F_{p_0}\left( \frac{D f(p)}{\abs{D f(p)}}\right):D^2 f(p)  \right\rangle \ge \frac{c_H}{T} \left( \lambda - \varepsilon \Lambda \right) >0. \]
This contradicts Proposition \ref{prop:equivalence} (ii). \\

\emph{Proof of Step 2:} The existence of $R_0$ as in the statement of Step 2 is a consequence of the implicit function theorem as in~\cite{White1987}, we report here the argument for the sake of completeness. Fix $q\in \partial \Omega$ and let $\nu_q \in \mathbb S^m$ be the inner normal of $\partial \Omega$ at $q$. Furthermore we fix an orthonormal basis $t_1, \dotsc, t_m$ spanning $T_q \partial \Omega = \nu_q^\perp \sim \R^m$ i.e. $\R^{m+1} = T_q\partial \Omega \times \operatorname{span} \nu_q$. We will write $(x,x_{m+1})$ for points in $T_q\partial \Omega \times \operatorname{span} \nu_q$. 

We consider the family of non-parametric functionals
\[ \mathcal{F}_r(x, u(x), D u(x)):= F\left(q + r (x, u(x)), (-D u(x), 1) \right) . \]
These are the non-parametric functionals associated to the image of the parametrized surfaces $x \mapsto q+ rx + r u(x) \nu_q$. 
Let $L_r$ be the Euler-Lagrange operators for $\mathcal{F}_r$. By strict convexity of $F$, planes are the unique minimizers for the frozen integrand $\nu \in \mathbb S^m \mapsto F(q, \nu )$. With respect to $\mathcal{F}_r$ this implies that the constant functions $u$ are the unique minimizers of $\mathcal{F}_0$, and in particular $L_0u=0$ for  every constant function $u$. The convexity of $F$ translates into the ellipticity of the linearization of $L_r$ around the constant $u_0=0$. Hence the implicit function theorem implies the existence of $\delta_q, R_q>0$  such that, for every couple of scalar functions $f,g$ with $\norm{f}_{C^{0,\alpha}}<\delta$, $\norm{g}_{C^{2,\alpha}} < \delta_q$, $U\subset B_2$ and $r\le R_q$, the boundary value problem 
\[ \left\{\begin{aligned}
L_r u &= f &&\text{ in } U \\
u &= g &&\text{ on } \partial U
\end{aligned}\right.\]
has a unique solution $u \in C^{2,\alpha}(U,\R)$ satisfying
\[ \norm{u}_{C^{2,\alpha}(U,\R)} \le C (\norm{f}_{C^{0,\alpha}} + \norm{g}_{C^{2,\alpha}} ).\]

The size of $R_q, \delta_q$ only depends on the $C^{2,\alpha}$ norm of $F$. Hence by compactness there exist $R_1, \delta_1>0$ such that $ \delta_q >\delta_1$ and $R_q >R_1$ for all $q \in \partial \Omega \cap B_{2R}$. \\
Let $H_F(q)$ as before denote the anisotropic mean curvature of $\partial \Omega$ with respect to the inner normal $\nu(q)$. Fix $0<R_0\le R_1$ such that 
\[ \max_{q \in \overline{B_R}\cap \partial \Omega } h_F(q)< \frac{\delta_1}{R_0}. \]
Now it is straight forward to check that $R_0$ has the desired properties. 

\end{proof}

We now show how to ``globalize'' the above  boundary estimate. We recall that for an \(F\)-stable surface it holds
\begin{equation}\label{eq:allard inequality} 
\int_M \phi^2 \abs{A}^2 \, d \H^2 \le c_1 \int_{M} \abs{D \phi}^2 + c_2 \phi^2 \, d \H^2,
 \end{equation}
for some constants $c_1(n,F)$, $c_2(n,F)>0$ and for all $\phi \in C^1_c(M)$, see \cite[Lemma 2.1]{Allard1983}  or ~\cite[Lamma A.5]{dephilippismaggi2}. 
 
\begin{Lemma}\label{lem:general_area_bound_in_dimension2}
Let $\Omega \subset \R^{3}$ and $\Gamma$ be a $C^{2,\alpha}$ embedded curve in $\partial \Omega \cap B_{2R}$ with $\partial \Gamma  \cap B_{2R} = \emptyset$. Furthermore let $M$ be a two dimensional $F$-stable, $C^2$ regular surface in $\Omega$ such that $\partial M \cap B_R = \Gamma$ and satisfying for some $0<C_0< \infty$ and $r_0 \le 1$
\begin{equation}\label{eq:assumed_area_bound} \sup_{\substack{q \in \Gamma \cap B_{R}\\ r< r_0}} \frac{ \H^2(M\cap B_r(q))}{\pi r^2} \le C_0. \end{equation}
 Then there exists a constant $C>0$ depending only on $F$ such that 
\begin{equation}\label{eq:general_area_bound} \sup_{\substack{B_r(p) \subset B_{R-r_0}\\ r< \frac{r_0}{3}; \dist(p, \Gamma) < \frac{r_0}{3}}} \frac{ \H^2(M\cap B_r(p))}{\pi r^2} \le C C_0. \end{equation}
\end{Lemma}
\begin{proof}
This lemma is a direct consequence of \eqref{eq:allard inequality} and of the extended monotonicity formula of L.~Simon (see \cite{Simon1993}).
Indeed, for every $p \in B_{R-r_0} \cap \Omega$ we fix $q \in \Gamma$ with 
\[ d:=\abs{q-p}=\dist(p, \Gamma)< \frac{r_0}{3}.\]
 Hence $q \in B_R\cap \Gamma$. If $\frac{d}{2}< r < \frac{r_0}{3}$, then $B_r(p) \subset B_{3r}(q)$ and we easily estimate
\[  \frac{ \H^2(M\cap B_r(p))}{\pi r^2} \le  \frac{9 \H^2(M\cap B_{3r}(q))}{\pi {(3r)}^2} \overset{\eqref{eq:assumed_area_bound}}{\le} 9 C_0. \]
If $r< \frac{d}{2}$ we argue as follows:   Fix a non-negative even function $\eta \in C^\infty(\R)$ with $\eta(t)=1$ for every $|t|\leq  \frac12$, $\eta(t)=0$ for $|t|\ge 1$ and $|\eta'(t)|\leq 3$ for every $t \in \R$. We choose $\phi(x):=\eta( \frac{\abs{x-p}}{d})$ in \eqref{eq:allard inequality} and, denoting with $H$ the isotropic mean curvature of $M$, we  obtain
\begin{equation}\label{stima utile}
\begin{split}
\frac12 \int_{B_{\frac{d}{2}\cap M}(p)} \abs{H}^2 \, d \H^2 &\le \int_{M
} \phi^2 \abs{A}^2 \\
&\le c_1 \int_{M}  \frac{1}{d^2}\abs{\eta'\Bigl(\frac{\abs{x-p}}{d}\Bigr)}^2 + c_2 \eta\left (\frac{\abs{x-p}}{d}\right)^2 \, d \H^2\\
&\le c_3 \frac{ \H^2(M \cap B_d(p))}{\pi d^2},
\end{split}
\end{equation}
where in the last inequality we used that $d<r_0\leq 1$.

Now we may use the extended monotonicity formula of L. Simon \cite[formula (1.3)]{Simon1993} to conclude that for any $r \le \frac{d}{2}$ we have for some universal constant $c>0$ (independent of $F, M, \Gamma$ and all our particular choices)
\begin{equation}\label{evviva}
 \frac{ \H^2(M\cap B_r(p))}{\pi r^2} \le c \left(  \frac{ \H^2(M\cap B_{\frac{d}{2}}(p))}{\pi d^2} + \int_{B_{\frac{d}{2}}(p)\cap M} \abs{H}^2 \, d \H^2 \right).
\end{equation}
Plugging \eqref{stima utile} in \eqref{evviva}, we conclude the lemma:
\[ \frac{ \H^2(M\cap B_r(p))}{\pi r^2} \le c (1 + 2c_3) \frac{ \H^2(M \cap B_d(p))}{\pi d^2} \le  4c(1+c_3) \frac{ \H^2(M \cap B_{2d}(q))}{\pi (2d)^2} \le C \,C_0,\]
where $C$ depends just on $F$.
\end{proof}

Now finally we can combine the obtained results with the curvature estimate in \cite{White1991} to prove Theorem \ref{thm:curvature estimates at the boundary}

\begin{proof}[Proof of \ref{thm:curvature estimates at the boundary}]
We choose $r_1 = \frac{r_0}{6}$ where $r_0$ is the radius in Proposition \ref{prop:no-boundary blow up}. Hence we may combine Proposition \ref{prop:no-boundary blow up} with Lemma \ref{lem:general_area_bound_in_dimension2} to deduce that for some constant $C$ depending only on $F, \Omega, \Gamma$
\[ \sup_{\substack{B_r(p) \subset B_{\frac{R}{2}}\\ r< 2r_1; \dist(p, \Gamma) < 2r_1}} \frac{ \H^2(M\cap B_r(p))}{\pi r^2} \le C. \]
In particular this implies that for each $q \in B_{\frac{R}{2}} \cap \Gamma$ we have 
\[ \sup_{\substack{B_r(p) \subset B_{2r_1}(q)}} \frac{ \H^2(M\cap B_r(p))}{\pi r^2} \le C .\]
Hence the triple $\Omega, M, B_{2r_1}(q)$ satisfies the assumptions of \cite[Theorem 5.2]{White1991} and we deduce that all principle curvatures of $M \cap B_{r_1}(q)$ are bounded by a constant depending only on $F, \Omega, \Gamma$.
\end{proof}

\bibliographystyle{siam}
\bibliography{areablowupbib}

\end{document}